\documentclass[12pt,thmsa, reqno]{amsart}

\let\oldtocsection=\tocsection

\let\oldtocsubsection=\tocsubsection

\renewcommand{\tocsection}[2]{\hspace{0em}\oldtocsection{#1}{#2}}
\renewcommand{\tocsubsection}[2]{\hspace{2em}\oldtocsubsection{#1}{#2}}

\usepackage{setspace}

\usepackage{amssymb, euscript,  mathrsfs}
\usepackage[dvips]{graphics}
\usepackage[title]{appendix}

\input{amssym.def}

\input{amssym.tex}

\usepackage{lineno}

\usepackage{tikz}
\usepackage[utf8]{inputenc}

\usetikzlibrary{matrix,arrows,decorations.pathmorphing}

\usepackage[all,cmtip]{xy}

\usepackage{amsmath}
\let\oldAA\AA
\renewcommand{\AA}{\text{\normalfont\oldAA}}

\def\Cal{\mathcal}

\def\C{{\Cal C}}

\def\S{{\Cal S}}

\def\L{{\Cal L}}

\def\bbr{{\Bbb R}}

\def\bbh{{\Bbb H}}

\def\bbc{{\Bbb C}}

\def\bbe{{\Bbb E}}
\def\bbs{{\Bbb S}}

\def\rnn{\bbr^{n+1}}

\def\det{{\hbox{\rm det}}}

\def\Pr{{\hbox{\rm Pr}}}

\def\rn{\bbr^n}

\def\part{\partial}
\def\intl{\int\limits}
\def\b{\beta}

\def\Gam{\Gamma}
\def\Om{\Omega}
\def\a{\alpha}
\def\om{\omega}

\def\Del{\Delta}
\def\del{\delta}
\def\vp{\varphi}

\def\gam{\gamma}

\def\Gam{\Gamma}

\def\sig{\sigma}
\def\lam{\lambda}
\def\z{\zeta}
\def\th{\theta}
\def\e{\varepsilon}
\def\t{\tau}

\def\chi{{\bf 1}}

\def\snm1{\bbs^{n-1}}

\font\frak=eufm10

\def\fr#1{\hbox{\frak #1}}

\def\intl{\int\limits}

\def\hn{\bbh^n}
\def\ch{\mathrm{cosh}}
\def\sh{\mathrm{sinh}}
\def\th{\mathrm{tanh}}

\def\Cs{\mathscr{C}}

\def\Cs{\mathscr{C c 1234}}

\def\cd{\stackrel{*}{\C}\!{}_{m, k}^\lam}
\def\sd{\stackrel{*}{\S}\!{}_{m, k}^\lam}
\def\cd0{\stackrel{*}{\C}\!{}_{m, k}^\lam}
\def\sd0{\stackrel{*}{\S}\!{}_{m, k}^\lam}

\def\ncd0{\stackrel{*}{\Cs}\!{}_{m, k}^\lam}

\newtheorem{theorem}{Theorem}[section]
\newtheorem{lemma}[theorem]{Lemma}

\theoremstyle{definition}
\newtheorem{definition}[theorem]{Definition}
\newtheorem{example}[theorem]{Example}

\theoremstyle{remark}
\newtheorem{remark}[theorem]{Remark}

\numberwithin{equation}{section}

\theoremstyle{corollary}
\newtheorem{corollary}[theorem]{Corollary}
\newtheorem{proposition}[theorem]{Proposition}
\newtheorem{conjecture}[theorem]{Conjecture}

\numberwithin{equation}{section}

\newcommand{\be}{\begin{equation}}
\newcommand{\ee}{\end{equation}}

\newcommand{\bea}{\begin{eqnarray}}
\newcommand{\eea}{\end{eqnarray}}
\newcommand{\Bea}{\begin{eqnarray*}}
\newcommand{\Eea}{\end{eqnarray*}}

\def\sideremark#1{\ifvmode\leavevmode\fi\vadjust{\vbox to0pt{\vss
 \hbox to 0pt{\hskip\hsize\hskip1em
\vbox{\hsize2cm\tiny\raggedright\pretolerance10000
 \noindent #1\hfill}\hss}\vbox to8pt{\vfil}\vss}}}%

\begin{document}

\title[Higher-Rank  Radon
Transforms]{ Higher-Rank  Radon
Transforms on Constant Curvature Spaces}

\author{ B. Rubin}

\address{Department of Mathematics, Louisiana State University, Baton Rouge,
Louisiana 70803, USA}
\email{borisr@lsu.edu}


\dedicatory{To the memory of Professor Nikolai Karapetovich Karapetyants
 on the occasion of his 80th birthday}



\begin{abstract}
We study higher-rank Radon transforms of the form $f(\t) \to \int_{\t \subset \z} f(\t)$, where $\t$ is a $j$-dimensional totally geodesic submanifold in the $n$-dimensional real constant curvature space and $\z$ is a similar submanifold of dimension $k >j$.
 The corresponding dual transforms are also considered. The transforms are explored the Euclidean case (affine Grassmannian bundles), the elliptic case (compact Grassmannians), and the hyperbolic case (the hyperboloid model, the Beltrami-Klein model, and the projective model). The main objectives are sharp conditions for the existence and injectivity of the Radon transforms
in  Lebesgue spaces,   transition  from one model to another,
support theorems,  and   inversion formulas. Conjectures and open problems are discussed.
 \end{abstract}

\maketitle




\tableofcontents


\section{Introduction }


\setcounter{equation}{0}
\noindent

Let $X$ be a real $n$-dimensional complete Riemannian manifold of constant sectional curvature $\kappa$.  One can realize  $X$ as the Euclidean space $\rn$  ($\kappa=0$), the unit sphere $S^n$ in $\bbr^{n+1}$ ($\kappa=1$), or the
 $n$-dimensional real hyperbolic space $\hn$ ($\kappa=-1$); see, e.g.,  \cite [p. 151]{H00}.     Let  $\Gam_X(n,j)$ and  $\Gam_X(n,k)$, $0 \le j<k\le n-1$,  be the families of totally geodesic submanifolds of $X$ of dimension $j$ and $k$, respectively ($j$-geodesics and $k$-geodesics, for short). The sets $\Gam_X(n,j)$ and  $\Gam_X(n,k)$  are homogeneous  spaces of the relevant Lie groups of motions and have the  structure of smooth manifolds; see, e.g., \cite [Chapter I, Sections 1 and 3]{H00}.
In the present paper, we consider a dual pair of
 Radon  transforms
 \be\label {jsewq}
 (R_X f)(\z)=\intl_{\t \subset  \z} f (\t), \qquad (R^*_X \vp)(\t)=\intl_{\z \supset  \t} \vp (\z),\ee
 where $f$ and $ \vp$ are complex-valued functions on $\Gam_X (n,j)$ and $\Gam_X (n,k)$, respectively,
 $ \t \in \Gam_X (n,j)$, $\z \in \Gam_X (n,k)$, and the
 integration is performed with respect to the corresponding canonical measures.

These transforms  are well-known for
 $j=0$ when $0$-geodesics are  points of $X$; see bibliographical notes  in
 \cite {H11, Ru15}.
 In the general case $0 \le j<k \le n-1$, similar higher-rank transforms were studied
 in the Euclidean and elliptic cases, i.e., for affine and compact Grassmannians,  by
 I.M. Gelfand,  S.G. Gindikin,  F. Gonzalez, M.I. Graev,  E.L. Grinberg, S. Helgason,  T.  Kakehi, G. Olafsson, E.E. Petrov,  B. Rubin, S. Sahi, Z.Ya. Shapiro,
R. Strichartz, Yingzhan Wang, Genkai Zhang, to mention a few. The  references  can be found  in  \cite{GK06, GR, Ru04d}.

The study of the higher-rank transforms (\ref{jsewq}) in the hyperbolic setting is pretty new. To the best of my knowledge, the only publications on this subject are due to Ishikawa \cite{I20, I21}. The  paper \cite{I20}
contains the description of the range of the operators (\ref{jsewq}) on the space of compactly supported smooth functions. The next paper \cite{I21}
 deals with injectivity and support theorems for more general Radon transforms on semisimple symmetric spaces. The results of \cite{I20, I21} are obtained  in the Lie-theoretic language in the framework of the Helgason double fibration theory.

 It is also worth noting that most of the authors mentioned above were  dealing  with $C^\infty$ compactly supported or Schwartz functions, while the properties of $R_X$ and $R^*_X$  on general Lebesgue spaces remained unexplored. On the other hand, these properties are very important for understanding the analytic nature of the operators $R_X$ and $R^*_X$ having a geometric (or group-theoretic) origin.

In the present paper, we develop a functional-theoretic (or coordinate) approach to the operators (\ref{jsewq}). This approach  differs essentially from that in \cite{I20, I21} (the hyperbolic case) and in 
\cite {GGR, GK03, GK04, GK06, GW, Gr85, Gr86, Ka, P1, Str86} (the Euclidean and elliptic cases).
We obtain explicit analytic formulas for $R_X$ and $R^*_X$ and establish sharp conditions under which these integrals  exist in the Lebesgue sense.
We also prove support theorems and present explicit inversion formulas in the cases when such formulas are not too complicated. For instance, simple inversion formulas are available on radial (or zonal) functions when the operators are essentially one-dimensional and expressed through
Abel type fractional integrals. Another simple case deals with functions $f$ belonging to the range of the classical totally geodesic
transform corresponding to $j=0$.

Some comments related to the hyperbolic case are in order. There exist several different models of $\hn$; see, e.g., \cite{CFKP, Ra}. 
The same results in different models look differently and are mutually beneficial.
 The  result in one model can sometimes be more transparent and  easier to obtain than its counterpart in another model. There exist simple formulas connecting points in different models; see, e.g., \cite [p. 71] {CFKP}.
 The corresponding formulas for totally geodesic submanifolds and the relevant Radon transforms require substantial  technical work related to computation of Jacobians;
 cf. \cite{BCK, Ku94}, \cite[pp. 412-416]{Ru15}.  In the present paper, we consider the hyperboloid model of $\hn$ as the basic one. Here we  follow
 \cite{GGV, Ru15, VK}. However, to obtain the main results, we also invoke the Beltrami--Klein  model (sometimes called the Klein  model, 
 the Cayley--Klein model, or the projective disk model). The  Beltrami--Klein  model is realized as  the interior of the unit ball with the relevant metric,
and the totally geodesic submanifolds  are represented by  chords of the corresponding dimension; see, e.g., \cite [pp. 7, 188] {Ra}.\footnote{See also Wikipedia, Beltrami--Klein model, \\ https://en.wikipedia.org/wiki/Beltrami\%E2\%80\%93Klein${}_{-}$model. } 

The main topics of the paper can be seen in the Contents. Explicit transition formulas (\ref{sform1}), (\ref{trou7}), (\ref{lity2}), (\ref{lity2q}), and (\ref{EPrtrou7}) from one model or setting
to another play a key role in our work. They pave the way to other new statements and can be viewed  as the main results of the paper.

 Throughout the paper, the letter $X$ in the above notation for the ambient space will be replaced by another one, depending on the context.

\subsubsection{ Some Notation} \label{Prel}

\noindent

In the following, $\rn$ is the usual $n$-dimensional real Euclidean space; $B_n$ is the open unit ball in $\rn$;  $S^{n -1}$  is the unit sphere  in $\rn$ with the surface area
$\sigma_{n-1}=2\pi^{n/2}/\Gamma(n/2)$;   $SO (n)$ stands for the special orthogonal group of $\rn$ with the relevant  Haar probability measure, so that
\be\label {sph}\intl_{S^{n-1}}f(\om)\,d\omega=\sigma_{n-1} \intl_{SO(n)}
f(\gamma \omega_0)\,d\gamma\ee
for any point $\omega_0 \in S^{n -1}$.

The letter $c$, sometimes with subscripts, is used for a constant that can be different at each occurrence. All functions are assumed complex-valued, unless otherwise  stated. Given a real-valued function $f$ and
 a complex number  $\lam$,  we set $f_+^\lam = f^\lam$ if $f>0$ and $f_+^\lam =0$, otherwise.

\section {The Euclidean Case and  Affine  Grassmannians}

In this section we recall some known facts from \cite{GK03, GK04, GK06, Ru04d} and prove new statements.  These results will be  extended later to the elliptic and hyperbolic cases.

\subsection {Definitions, Duality, Radial Functions}\label {deta}

Let $G_{n,d}$ will be the  Grassmann manifold
 of  $d$-dimensional linear subspaces  of $\rn$, $n\ge 2$. The corresponding Grassmannian bundle  (or affine Grassmannian)
 of  $d$-dimensional affine planes in $\rn$ will be denoted by ${\rm Gr} (n,d)$. We recall that
\be\label {ewhich} \dim {\rm Gr} (n,d)=(d+1)(n-d),\ee
 and  assume
 $0\le d \le n-1$,  where the case $d=0$ corresponds to points in $\rn$.


The space ${\rm Gr} (n,d)$ is a homogeneous manifold of the group $E(n)$ of isometries of $\rn$.
Local coordinates on ${\rm Gr} (n,d)$ are described in \cite{Ri}.
Elements of  ${\rm Gr} (n,d)$ can be parametrized in different ways. For example, each plane
  $\t \in {\rm Gr} (n,d)$ can be   associated with a pair
$(\xi, u)$, where $\xi \in G_{n,d}$ and $ u \in \xi^\perp$, the
orthogonal complement to $\xi $ in $\rn$. Abusing notation, we denote by $ |\t|$ the
 Euclidean distance from $\t \equiv \t( \xi, u)$ to the origin $o$ of $\rn$.
Clearly, $|\t|=|u|$, the Euclidean norm of $u$. We denote
\be\label {dfwhich}
\tilde {\rm Gr} (n,d)=\{ \t \in {\rm Gr} (n,d): |\t| > 0\}.\ee
A function $f$ on ${\rm Gr} (n,d)$  is called   radial
if it is  $SO(n)$-invariant. Every such function has the form $f(\t) \!= \!f_0 (|\t|)$ for some single-variable function $f_0$.

The manifold
${\rm Gr} (n,d)$ will be endowed with the product measure $d\t=d\xi du$,
where $d\xi$ is the
 $SO(n)$-invariant probability measure  on $G_{n,d}$  and $du$ denotes the usual volume element in $\xi^\perp$. If $f(\t) = f_0 (|\t|)$, then, passing to polar coordinates in $\xi^\perp$, we obtain
 \be\label {ewadch}
 \intl_{{\rm Gr} (n,d)}\!\!\! f(\t) \, d\t= \sig_{n-d-1}\intl_0^\infty f_0 (r) \, r^{n-d-1}\, dr.\ee

 The spaces $C({\rm Gr} (n,d))$  and $L^p({\rm Gr} (n,d))$ of continuous and $L^p$ functions are defined in a standard way;
 $C_c^\infty({\rm Gr} (n,d))$ is the space of compactly supported $C^\infty$ functions on  ${\rm Gr} (n,d)$.
 The notation $S({\rm Gr} (n,d))$ will be used for the Schwartz space of rapidly decreasing smooth functions defined  in
  \cite{Ri}; see also \cite {Go90}. We will also  deal with weighted Lebesgue spaces
\be \label {eads}
L^1_\lam
({\rm Gr} (n,d)) =\Bigg \{ \; f \, : ||f||=\intl_{ {\rm Gr} (n,d)} \frac {|f(\t)| \;
 d\t}{(1+|\t|)^\lam} < \infty  \; \Bigg \}. \ee


 Given  a pair of  affine Grassmannians  ${\rm Gr} (n,j)$ and ${\rm Gr} (n,k)$, $0\le j<k\le n-1$,
  we set
  \be\label {which} \t= \t( \xi, u)\in {\rm Gr} (n,j), \qquad \xi \in G_{n,j}, \quad u \in \xi^\perp; \ee
\be\label {which1}  \z= \z(\eta, v) \in {\rm Gr} (n,k), \qquad \eta \in G_{n,k}, \quad v \in \eta^\perp. \ee
The Radon transform  of a function $f: {\rm Gr} (n,j) \to \bbc$ is a
function $R_A f: {\rm Gr} (n,k) \to \bbc$ defined by
\be\label {lasz} (R_A f)(\z)  = \intl_{\t
\subset \zeta} \! f(\t)\, d_{\z}\t \equiv \intl_{\xi \subset \eta} d_\eta \xi
\intl_{\xi^\perp \cap  \eta} f(\xi, v+y)\, dy. \ee
Here $d_\eta \xi$
denotes the standard probability measure on the Grassmannian of
all $j$-dimensional linear subspaces  of $\eta$, and the right-hand
side gives precise meaning to the integral $\int_{\t
\subset \zeta} f(\t)\, d_{\z}\t$. The case $j=0$  yields
 the Radon-John $k$-plane transform which integrates $f\equiv f (x)$, $x\in \rn$, over $k$-dimensional planes in $\rn$;
 cf. \cite [Chapter 2]{GGG}, \cite{Go90}, \cite [Chapter I, Section 6]{H11}, \cite{Ke, Ri, Ru04b}, to mention a few.

For the following, it is convenient to split the coordinate unit vectors $e_1, \ldots, e_{n}$ in three groups
\be\label {krtg}
\{e_1, \ldots, e_{n-k}\}, \quad \{e_{n-k+1}, \ldots, e_{n-j}\}, \quad \{e_{n-j+1}, \ldots, e_{n}\},\ee
and establish notation for the corresponding coordinate planes:
\be \label {krtg1}
\bbr^j =\bbr e_{n-j+1} \oplus\cdots\oplus \bbr e_{n}, \qquad \bbr^k =\bbr e_{n-k+1}\oplus\cdots\oplus  \bbr e_{n},\ee
\be\label {krtg2}
\bbr^{n-j}=\bbr e_1\oplus\cdots\oplus  \bbr e_{n-j}, \qquad \bbr^{n-k}=\bbr e_1\oplus\cdots\oplus  \bbr e_{n-k},\ee
\be\label {aakrtg2}
 \bbr^{k -j}= \bbr e_{n-k+1}\oplus\cdots\oplus  \bbr e_{n-j}.\ee
In the cases $j=0$, $k=n-1$, and $j=k-1$ the changes  are obvious.

If $g\in  SO(n)$ is a rotation satisfying
\[ g: \bbr^k \to \eta, \qquad g: e_{n-k}
\to v/|v|,  \]
then, denoting $f_g(\t)=f(g\t)$, one can write (\ref{lasz}) as
    \bea (R_A f)(\eta, v)&=&\intl_{G_{k,j}} d\sig \intl_{\sig^\perp \cap \bbr^k}
  f_g (\sig, |v|e_{n-k}+y) \, dy \nonumber\\
\label {lasz1} &=&\intl_{SO(k)} d\gam \intl_{\bbr^{k -j}}
  f_g (\gam (\bbr^{j}+ |v|e_{n-k}+z)) \, dz. \quad \eea

The dual Radon transform  of a function $\vp(\z) \equiv \vp(\eta,
v)$ on ${\rm Gr} (n,k)$ is a function $(R^*_A \vp)(\t)\equiv (R^*_A \vp)(\xi, u) $  on ${\rm Gr} (n,j)$,
defined by
\be\label {lasz2}  (R^*_A \vp)(\t)
= \intl_{\z \supset \t} \vp(\z)\,  d_\t \z= \intl_{\eta  \supset \xi}
\vp(\eta +u) d_\xi \eta=\intl_{\eta \supset \xi}
\vp(\eta,\Pr_{\eta^\perp} u) d_\xi \eta. \ee
Here
$\Pr_{\eta^\perp} u$ denotes the orthogonal projection of $u \;
(\in \xi^\perp)$ onto $\eta^\perp (\subset \xi^\perp)$, $d_\xi
\eta$ is the relevant normalized measure.  This transform
integrates  $\vp$ over all $k$-planes $\z$   containing  the
$j$-plane $\t$. To give (\ref{lasz2}) precise meaning,
 we choose
a rotation $g_\xi \in  SO(n)$ satisfying $g_\xi \bbr^j=\xi$, and let
$SO(n-j)$ be the subgroup of rotations in the coordinate plane
$\bbr^{n-j}$.  Then (\ref{lasz2}) means \be \label {laszcv} (R^*_A \vp)(\t) \equiv  (R^*_A \vp)(\xi, u)=
\intl_{SO(n-j)} \vp(g_\xi \rho \bbr^k +u) \, d\rho, \ee
where the integral on the right-hand side does not depend on the afore-mentioned choice of $g_\xi$.

\begin{proposition} {\rm \cite[Lemma 2.1]{Ru04d},    \cite [p. 69]{H11}} The equality
\be \label {wbcdaff} \intl_{{\rm Gr} (n,k)} (R_A f)(\z) \,\vp(\z)\, d \z=\intl_{{\rm Gr} (n,j)} f(\t)\, (R^*_A \vp)(\t)\, d\t \ee
holds provided
 that  either side of it is finite when  $f$ and $\vp$ are replaced by $|f|$ and $|\vp|$,
respectively.
\end{proposition}

Because the Radon transforms $R_A$ and $R^*_A$  commute with rotations, the functions $R_A f$ and $R^*_A \vp$ with $f$ and $\vp$ radial  (i.e. $SO(n)$-invariant)
 have a simple representation in terms of Erd\'{e}lyi--Kober
type fractional integrals. Information about such integrals can be found in Appendix (see Section \ref{kuku}).
\begin{proposition}\label {nsfo5} {\rm \cite[Lemma 2.3]{Ru04d}} For $\t \in {\rm Gr} (n,j)$ and $\z \in {\rm Gr} (n,k)$,
let $f(\t)=f_0 (|\t|)$, $\vp(\z)= \vp_0 (|\z|)$. Then $(R_A f)(\z)= F_0(|\z|)$, $(R^*_A \vp)(\t)= \Phi_0(|\t|)$, where
\be \label{besg} F_0(s)= \sig_{k-j-1}\intl_s^\infty f_0(r) (r^2 -s^2)^{(k-j)/2
-1} r dr, \quad s=|\z|,\ee
\be \label{besg1} \Phi_0(r)= \frac{\sig_{k-j-1} \,
\sig_{n-k-1}}{\sig_{n-j-1} \, r^{n-j-2}}\intl_0^r \vp_0 (s) (r^2
-s^2)^{(k-j)/2 -1} s^{n-k-1} ds, \quad r=|\t|,\ee
provided that the corresponding
integrals exist in the Lebesgue sense.
\end{proposition}

\begin{remark}\label {nusfo}
 The integral (\ref{besg}) is finite for almost all $s>0$ provided that
\be\label {for10zf} \intl_a^\infty |f_0 (r)|\, r^{k-j-1} dr <\infty\ee
 for all $a>0$. If $f$ is non-negative and (\ref{for10zf}) fails for some $a>0$, then $F_0 (s)=\infty$ for every $s\!\ge \!0$. The proof of these facts can be found in  \cite [Lemma 2.42]{Ru15}.
\end{remark}

\begin{remark} Formulas (\ref{besg}) and (\ref{besg1}) reveal a remarkable fact, which is obvious for  all Abel-type
fractional integrals. Specifically, to provide the existence of $F_0(s)$ for $s>0$, we need to take
care of the behavior of $f_0 (r)$ at infinity and may ignore its behaviour as $r\to 0$. On the other hand, in (\ref{besg1}),
the behavior of $\vp_0 (s)$ as $s\to 0$ is crucial, while no restriction is needed when $s\to \infty$. We say that these
integrals have a {\it one-sided structure}, {\it the right-sided} for (\ref{besg}) and {\it the left-sided}   for (\ref{besg1}).
Owing to the $SO(n)$-equivariance, this one-sided structure  still holds for  $R_A f$ and $R^*_A \vp$  with arbitrary, not necessarily radial $f$ and $\vp$, as can be seen below in the corresponding existence conditions and the support theorems. This important phenomenon will be taken into account in our study.
\end{remark}

\subsection {Existence in the Lebesgue Sense}

Our next concern is convergence of the integrals $R_A f$ and $R^*_A \vp$.
We begin with some known facts.
\begin{proposition} \label {thers} {\rm \cite [p. 254]{GK03}}   Let  $ 0 \le j<k\le n-1$.  The Radon transform
$R_A$ acts from $S({\rm Gr} (n,j))$ to $S({\rm Gr} (n,k))$ and from $C_c^\infty({\rm Gr} (n,j))$ to $C_c^\infty({\rm Gr} (n,k))$.
\end{proposition}

To extend this statement to non-smooth functions, we invoke the Lebesgue weighted spaces (\ref {eads}).

\begin{proposition} \label {ther} {\rm \cite [Corollary 2.9]{Ru04d}} Let  $ 0 \le j<k\le n-1$, $\lam =n-k$. The Radon transform $R_A$ is
a linear bounded operator from $L^1_\lam ({\rm Gr} (n,j))$ to
$L^1_{\lam +\del} ({\rm Gr} (n,k))$ for all $\del >0$.
The value
$\lam =n-k$ is best possible in the sense that for any $\e>0$ there exists $f \in
L^1_{n-k+\e} ({\rm Gr} (n,j))$ such that $(R_A f)(\z) \equiv \infty$.
\end{proposition}

\begin{proposition}\label{L2} {\rm \cite[Corollary 2.6]{Ru04d}}
 If
 \[ f\in L^p ({\rm Gr} (n,j)), \qquad 1 \le p < (n-j)/(k - j),\] then
$(R_A f)(\z)$ is finite for almost all $\z \in {\rm Gr} (n,k)$.
 If $p \ge (n-j)/(k -j)$ and \[f(\t)=(2+|\t|)^{(j-n)/p} (\log (2+|\t|))^{-1} \quad (\in L^p ({\rm Gr} (n,j))),\]
 then $(R_A f)(\z) \equiv \infty$. In particular, if $f \in C ({\rm Gr} (n,j))$ satisfies
 $f(\t)=O(|\t|^{-\lam})$ as $|\t| \to \infty$, then $(R_A f)(\z)$ is finite for all $\z \in {\rm Gr} (n,k)$  provided $\lam>k-j$, and can be identically infinite if $\lam \le
 k-j$.
\end{proposition}

By Proposition \ref {ther},  if $f\in L^1_{n-k} ({\rm Gr} (n,j))$, then   $(R_A f)(\z)$ is finite for almost all $\z \in {\rm Gr} (n,k)$. The assumption $ f=f (\t)\in L^1_{n-k} ({\rm Gr} (n,j))$
is  sharp for $|\t| \to \infty$ but  is redundant for $|\t| \to 0$. The next new statement gives more precise information.

\begin{lemma}\label {lrrb1}   Let  $ 0 \le j<k\le n-1$.   If
\be\label{for10zk} \intl_{|\t|>a} \frac{|f(\t)|}{|\t|^{n-k}}\, d\t <\infty\ee
  for all $a>0$, then $(R_A f)(\z)$ is finite for almost all $\z\!\in \!{\rm Gr} (n,k)$.
If  $f$ is nonnegative, radial,  and (\ref{for10zk}) fails for some $a>0$, then $(R_A f)(\z)\!=\!\infty$ for every $\z\!\in\! {\rm Gr} (n,k)$.
\end{lemma}
\begin{proof} It suffices to assume $f\ge 0$. By  the $SO(n)$-invariance of the measure $d\z$, for any $\e>0$ we have
\[
I\equiv \intl_{|\z|>a}\frac{(R_A f)(\z)}{|\z|^{n-k+\e}} \, d\z=    \intl_{SO(n)} d\gam \intl_{|\z|>a}\frac{(R_A f)(\gam \z)}{|\z|^{n-k+\e}}\,d\z=\intl_{|\z|>a}\frac{(R_A f_{ave})(\z)}{|\z|^{n-k+\e}}\,d\z\] where
$f_{ave} (\t)=\int_{SO(n)} f(\gam \t)\, d\gam \equiv  f_0 (|\t|)$ is the $SO(n)$-average of $f$. Hence, by (\ref{besg}) and (\ref{ewadch}),
\bea
I\!\!\! &=&\!\! \! c \intl_a^\infty \frac{ds}{s^{1+\e}} \intl_s^\infty \!f_0(r) (r^2 \!-\!s^2)^{(k-j)/2-1} r \,dr\nonumber\\
\!\!\! &=&\!\!\! c \intl_a^\infty \! f_0(r) r \,dr \!\intl_a^r  (r^2 \!-\!s^2)^{(k-j)/2-1} \frac{ds}{s^{1+\e}}\nonumber\\
\!\!\! &=&\!\! \! c \intl_a^\infty \! f_0(r) r^{k-j-1-\e}\psi (r)\, dr, \; \psi (r)\! =\!\intl_{a/r}^1  (1\! -\!t^2)^{(k-j)/2-1}\,\frac{dt}{t^{1+\e}} =O(r^\e),\; r \to \infty.\nonumber\eea
Then
\[
I\le c_1 \intl_a^\infty \! f_0(r) r^{k-j-1}\,dr=c_2 \intl_{|\t|>a} \frac{|f_0(|\t|)|}{|\t|^{n-k}}\, d\t=c_2 \intl_{|\t|>a}  \frac{|f(\t)|}{|\t|^{n-k}}\, d\t.\]
The last integral is finite by the assumption. Hence, because $a>0$ is arbitrary, it follows that $(R_A f)(\z)$ is finite for almost all $\z\!\in \!{\rm Gr} (n,k)$.

To prove that  (\ref{for10zk}) is sharp, we observe that if it fails for a  nonnegative  radial function  $f(\t)\equiv f_0 (|\t|)$ and  some $a>0$, then
\[\intl_a^\infty f_0 (r)\, r^{k-j-1} dr =\infty.\]
 Hence  the fractional integral  $F_0 (s)$ in (\ref{besg}) diverges  for all $s>0$ (see \cite [Lemma 2.42]{Ru15}), and therefore $(R_A f)(\z)\!=\!\infty$ for all $\z\!\in\! {\rm Gr} (n,k)$.
\end{proof}

The next statement gives a simple sufficient condition of the pointwise convergence of the integral   $(R_A f)(\z)$.

\begin{proposition} \label{Arges} If
\be\label {rges}A \equiv {\rm ess} \!\!\!\!\sup\limits_{ \t \in {\rm Gr} (n,j)} |\t|^\lam |f(\t)|<\infty\quad \text {for some $\; \lam > k-j$},\ee
 then  $(R_A f)(\z)$ is finite for  all $|\z|>0$. The condition  $\lam > k-j $  is sharp.
\end{proposition}
\begin{proof} Let $f_1 (\t) =|\t|^{-\lam}$. Then, by (\ref{besg}),
\bea |(R_A f)(\z)|&\le& A (R_A f_1)(\z)= A\sig_{k-j-1}\intl_{|\z|}^\infty \!r^{-\lam} (r^2 - |\z|^2)^{(k-j)/2-1} r dr\nonumber\\
&=& c\, |\z|^{k-j-\lam}<\infty,\nonumber\eea
as desired.
\end{proof}

\begin{proposition}\label {lrrbx} Let  $ 0 \le j<k\le n-1$. The dual transform  $(R^*_A \vp)(\t)$ is finite a.e.
 for every locally integrable function $\vp $ on ${\rm Gr} (n,k)$ and represents a locally integrable function on ${\rm Gr} (n,j)$. Further,
  if
 \be\label {rgesB} B\equiv{\rm ess} \!\!\!\!\sup\limits_{ \z \in {\rm Gr} (n,k)} |\z|^\del |\vp(\z)|<\infty \quad \text {for some $\;\del < n- k$},\ee
 then  $(R^*_A \vp)(\t )$ is finite for  all $|\t|>0$. The condition  $\del < n- k$  is sharp.
 \end{proposition}
\begin{proof}
The first statement follows from the equality
\be\label{lab8}
\intl_{|\t|<a} (R^*_A \vp)(\t) \, d\t =
\frac{\pi^{(k-j)/2}}{\Gam (1+(k -j)/2)}
\intl_{|\z|<a} \vp(\z) \,(a^{2}   -
|\z|^{2})^{(k - j)/2} \,d \z,
\end{equation}
which is a particular case of \cite [formula (2.19] {Ru04d}. If $\vp$ satisfies  (\ref{rgesB}) and $\vp_1 (\z) =|\z|^{-\del}$, then by (\ref{besg1}),
\bea |(R^*_A \vp)(\t)| &\le& B (R^*_A \vp_1)(\t)= \frac{c\, B}{r^{n-j-2}}\intl_0^{|\t|}  (|\t|^2-s^2)^{(k-j)/2 -1} s^{n-k-\del-1} \,ds\nonumber\\
&=& c_1\, |\t|^{-\del}<\infty\nonumber\eea
 if $\del <n-k$. This  gives the result.
\end{proof}

\subsection {Injectivity}

The following proposition  is due to Gonzalez and Kakehi  \cite [Theorem 6.4]{GK03}, \cite [Lemma 5.1]{GK04}.
\begin{proposition} \label {thgrer1}
Let  $ 0 \le j<k\le n-1$. If $ j+k \le n-1$, then  the Radon transform $R_A$ is an injective operator from $S({\rm Gr} (n,j))$  to $S({\rm Gr} (n,k))$.
If  $ j+k \ge n$, then the dual Radon transform $R^*_A$ is an injective operator from $C^\infty ({\rm Gr} (n,k))$  to $C^\infty ({\rm Gr} (n,j))$.
\end{proposition}

Note that the case $ j+k =n-1$ is not included in the second statement; see Conjecture \ref {don}  below.

\begin{proposition} \label {ther1} {\rm \cite [Theorems 4.2, 4.4]{Ru04d}}
Let  $ 0 \le j<k\le n-1$.  The Radon transform $R_A$ is injective on  $L^1_{n-k} ({\rm Gr} (n,j))$  if and only if $ j+k \le n-1$.
The dual Radon transform $R^*_A$  is
injective on  $L^1_{j+1} ({\rm Gr} (n,k))$   if and only if $ j+k \ge n-1$. Thus, if $ j+k = n-1$, then both  transforms
are injective simultaneously on the corresponding spaces.
\end{proposition}

The following useful result does not cover the case $ j+k = n-1$. However, it extends the second statement of Proposition \ref{ther1} to
functions $\vp$ that allow  arbitrary power growth at infinity and eliminates the smoothness assumption in the second statement in Proposition \ref{thgrer1}.

\begin{theorem} \label {rther1} Let $ 0 \le j<k\le n-1$, $j+k\ge n$.
 The dual Radon transform $R^*_A$  is
injective on the set of functions $\vp: {\rm Gr} (n,k) \mapsto \bbc$ satisfying
\be\label {consi}
I=\intl_{G_{n, n-k}}  |\vp(\tilde \eta^\perp,\Pr_{\tilde \eta} u)|\, d\tilde \eta <\infty \quad \text {for each  $u\in \rn \setminus {0}$},\ee
$\Pr_{\tilde \eta} u$ being the orthogonal projection of $u$ onto $\tilde \eta$.  In particular,  (\ref{consi}) holds  for functions $\vp$ satisfying (\ref{rgesB}).
\end{theorem}

\begin{proof}  We set $\tilde \xi=\xi^\perp\in  G_{n, n-j}$,  $\tilde \eta=\eta^\perp \in G_{n, n-k}$,  and
write the last integral in (\ref{lasz2}) as
\[(R^*_A \vp)(\xi, u)\!=\!
\intl_{\eta \supset \xi}\!
\vp(\eta,\Pr_{\eta^\perp} u) d_\xi \eta=\!\!
\intl_{\tilde \eta \subset \tilde \xi}
\!\!\!\psi_u (\tilde \eta) d_{\tilde \xi} \tilde \eta,\qquad \psi_u (\tilde \eta)= \vp(\tilde \eta^\perp,\Pr_{\tilde \eta} u).\]
Suppose that $u\in \rn \setminus {0}$ is fixed and $\xi \in u^\perp$. Then
 $(R^*_A \vp)(\xi, u)$ can be thought of as a restriction of the higher-rank Funk-Radon transform $F \psi_u$ onto the $(n-j)$-dimensional linear subspace
$\tilde \xi=\xi^\perp$ containing $u$. The assumption (\ref{consi}) is equivalent to $\psi_u \in L^1 (G_{n, n-k})$. It is known that  $F$ is injective on $L^1 (G_{n, n-k})$ provided
that $(n-i) + (n-k) \le n$, that is, $j+k\ge n$; see \cite [Theorem 1.2]{GR}. Because $u\in \rn \setminus {0}$ is arbitrary, the first statement follows.

To prove the second statement, it remains to estimate the integral $I$.
For each  $u\in \rn \setminus {0}$,
\[
I\le c_u \intl_{G_{n, n-k}}  |\Pr_{\tilde \eta} u|^{-\del}\, d\tilde\eta= c_u \intl_{{\rm St} (n,n-k)}  |\om\om' u|^{-\del}\, d\om=c_u\, I_0,\]
where $c_u$ is a constant depending on $u$, ${\rm St} (n,n-k)$ is the Stiefel manifold of orthonormal $(n-k)$-frames in $\rn$, $d\om$ is the relevant
properly normalized Haar measure.  Each
$\om \in {\rm St} (n,n-k)$ is an $n\times (n-k)$ real matrix satisfying  $\om'\om =I_{n-k}$, where $\om'$ is the transpose of $\om$ and $I_{n-k}$
is the identity $(n-k)\times (n-k)$ matrix. The column-vectors of $\om$ form an orthonormal basis of $\tilde\eta \in G_{n, n-k}$ and $\Pr_{\tilde \eta}= \om\om'$.
 Because the orthogonal group $O(n)$ acts on  ${\rm St} (n,n-k)$ transitively, setting $\om_0= \left[\begin{array} {cc} 0 \\ I_{n-k}  \end{array} \right]$
 and changing variables, we have
\[
I_0=|u|^{-\del} \!\intl_{O(n)}  \!\! |\om_0\om'_0 \gam e_n|^{-\del}\, d\gam= \frac{|u|^{-\del}}{\sig_{n-1}} \intl_{S^{n-1}} \!\! |E_0 \theta|^{-\del}\, d\theta,\quad  E_0 =\left[\begin{array}{cccc} 0 & 0   \\
                              0 & I_{n-k}   \end{array} \right].\]
Passing to bi-spherical coordinates (see, e.g. \cite [Lemma 1.38]{Ru15})
\be \label {cas6}\theta= \left[\begin{array} {cc} a\,\cos \a \\ b\,\sin \a \end{array} \right], \quad a\in S^{k-1}, \quad b\in S^{n-k-1},\quad
0\le \a\le \pi/2,\ee
 we continue:
\[
I_0=\frac{|u|^{-\del}\, \sig_{k-1}  \sig_{n-k-1}}{\sig_{n-1}} \intl_0^{\pi/2}
\cos^{k-1}\a\,\sin^{n-k-\del-1}\a \,d\a <\infty\]
whenever $\del<n-k$. This completes the proof.
\end{proof}

Let us show that an  analogue of Theorem \ref{rther1} for the Radon transform $R_A$ holds if $j+k\le n-2$. The transition from $R^*_A$ to $R_A$ can be performed by making use of
 the  Kelvin-type
quasi-orthogonal  inversion transformation introduced in \cite[Section 5]{Ru04d}.  Specifically,
for $\t \equiv \t(\xi, u) \in \tilde {\rm Gr} (n,j)$ (see (\ref{dfwhich})), we denote by $\{\t\} \in G_{n, j+1}$
the smallest linear subspace containing $\t$, and set
\[
\tilde{\xi}=\{\t\}^{\perp} \in G_{n, n-j-1}, \quad\tilde{u}= -
\frac{u}{|u|^{2}} \in\tilde{\xi}^{\perp}, \quad\tilde{\t}
\equiv\tilde{\t}(\tilde{\xi},\tilde{u})\in \tilde {\rm Gr} (n, n-j-1).
\]
Consider the Kelvin-type mapping
\be\label {Kelvin} \tilde {\rm Gr} (n,j) \ni\t \xrightarrow{\quad\varkappa\quad}
\tilde{\t}\in \tilde {\rm Gr}(n, n-j-1).
\ee
Clearly, $\varkappa \, (\varkappa \, (\t))=\t$ and $|\t|=|
\tilde{\t}|^{-1}$. In a similar way, for
$\z\equiv\z(\eta, v)\in  \tilde{\rm Gr} (n,k)$, $v\neq0$, we denote
\[
\tilde{\eta}=\{\z\}^{\perp} \in G_{n, n-k-1}, \quad\tilde{v}=
-\frac{v}{|v|^{2}} \in\tilde{\eta}^{\perp}, \quad\tilde{\z}
\equiv\tilde{\z}(\tilde{\eta},\tilde{v})\in \tilde {\rm Gr} (n, n-k-1),
\]
and get
\begin{equation} \tilde{\rm Gr} (n,k) \ni\z\xrightarrow{\quad\varkappa\quad}
\tilde{\z}\in \tilde{\rm Gr} (n, n-k-1).
\end{equation}

\begin{definition}
Let $R_A: f(\t) \mapsto(R_A f)(\z)$ be the Radon transform
 that takes functions on $\tilde {\rm Gr} (n,j)$ to functions on
$\tilde{\rm Gr} (n,k), \; k>j$. If
\[\tilde{\t}=\varkappa\,(\t)\in \tilde{\rm Gr} (n, n-j-1), \qquad \tilde{\z}=\varkappa\,(\z)\in \tilde {\rm Gr} (n, n-k-1),\]
 then the associated
Radon transform $ \tilde R_A: \tilde{f}(\tilde{\z}) \to (\tilde R_A
\tilde{f})(\tilde{\t})$ that integrates $\tilde{f} $ over all
$\tilde{\z} \subset \tilde{\t}$ is called
\textit{quasi-orthogonal} to $R_A$.
\end{definition}

\begin{proposition}\label {ans}
\label{exath} {\rm \cite[Theorem 5.5]{Ru04d}} Let $0\le j<k \le n-1$,
\be\label {adon} (U\Phi)(\t)\!=\!c \,|\t|^{k-n} (\Phi \circ \varkappa)(\t), \quad (V\vp)(\tilde\z)\!=\! |\tilde \z|^{j-n} (\vp\circ  \varkappa^{-1})(\tilde\z),\ee
$c=\sig_{n-k -1}/\sig_{n-j-1}$. Then
\be\label {sform1}
R_A^{*} \vp= U\tilde R_A V\vp,\ee
provided that either side of this equality exists in the Lebesgue sense.
\end{proposition}

One can show (see \cite [Theorem 5.5 (i)]{Ru04d}) that
\be\label{ildes}
\intl_{{\rm Gr} (n,k)}\frac{\vp (\z)\, d\z}{(1+ |\z|^2)^{(j+1)/2}}=\frac{\sig_{n-k-1}}{\sig_k} \intl_{{\rm Gr} (n,n-k-1)}\frac{(V\vp) (\tilde{\z})\, d\tilde{\z}}{(1+ |\tilde{\z}|^2)^{(j+1)/2}}.\ee

From (\ref{sform1}) it follows that $R_A^{*}$ is injective on  the set
\[L^\infty_\del ({\rm Gr} (n,k))= \{ \vp:  B \equiv {\rm ess} \!\!\!\!\sup\limits_{ \t \in {\rm Gr} (n,k)} |\z|^\del |\vp(\z)|<\infty \; \text {for some $\;\del < n- k$}\},\]
(cf. (\ref{rgesB})) if and only if $\tilde R_A$ is
injective on the range $V(L^\infty_\del ({\rm Gr} (n,k)))$. The latter coincides with the similar set
$L^\infty_\lam ({\rm Gr} (n,j_1))$, $\lam >k_1-j_1$, where
\[j_1 =n-k-1, \qquad k_1 =n-j-1, \qquad \lam=n-j-\del,\]
$k_1-j_1=k-j$.
 Changing notation, we conclude that Theorem \ref{rther1} yields the following  ``dual'' statement for the operator $R_A$.

\begin{theorem} \label {Arther1} Let $ 0 \le j<k\le n-1$, $j+k\le n-2$. Then the
 Radon transform $R_A$  is
injective  on the set of functions satisfying
\[A\equiv {\rm ess} \!\!\!\!\sup\limits_{ \t \in {\rm Gr} (n,j)} |\t|^\lam |f(\t)|<\infty  \quad \text {for some $\;\lam > k-j$},\]
where the condition $\lam >k -j$ is sharp.
\end{theorem}

We observe that the injectivity conditions in Theorems \ref{Arther1} and \ref{rther1}  coincide with the existence conditions in Propositions   \ref{Arges}  and \ref{lrrbx}, respectively.

\begin{conjecture}\label {don} {\rm We conjecture that if $j+k= n-1$, then both $R_A$ and $R^*_A$ are non-injective on $C^\infty (\tilde {\rm Gr} (n,j))$ and $C^\infty (\tilde {\rm Gr} (n,k))$,
 respectively. This fact is known for $j=0$, i.e.,
for the  hyperplane Radon transform in $\rn$ and its dual; see} {\rm \cite [Theorems 4.112, 4.113]{Ru15}}.
\end{conjecture}

\subsection {Inversion Formulas}
It is natural to expect that if $R_A$ and $R^*_A$ are injective on some classes of functions then there exist the corresponding inversion formulas.
Such formulas may have different analytic form and depend on the class of functions. We are not familiar with any universal reconstruction
formulas for $R_Af$ or $R^*_A\vp$ that would cover all injectivity cases. Each new formula of such a kind would be of interest.
Most of the known inversion formulas  look pretty complicated and we do not present them here. These formulas can be simplified
if  $f$  or $\vp$  enjoy some additional symmetry (e.g., $f$ and $\vp$ are radial). More information on this subject, including particular cases,  can be found in
 \cite [Theorems 4.2, 4.4]{Ru04d}, \cite{RW17, RW18}.

 \subsection {Support Theorems}

 The following result is due to Gonzalez and Kakehi \cite [Theorem 6.2, Corollary 7.1]{GK06}; see also  \cite [pp. 11, 33]{H11} for the case $j=0$.
\begin{proposition} \label {thrt4}  Let  $0\le j<k \le n-1$, $j+k\le n-1$,
 $\t \in {\rm Gr} (n,j)$,  $\z \in {\rm Gr} (n,k)$, $r>0$. If
 $f\in S({\rm Gr} (n,j))$, then
\be\label{kvv1}
(R_A f)(\z)=0\; \forall \, |\z| >r\; \Longrightarrow \; f(\t)=0\; \forall \, |\t| >r.\ee
\end{proposition}

Our aim is to extend this statement to wider classes of functions and include the dual transform into the consideration.

It is convenient to distinguish the case  $j+k\le n-1$ (or $\dim {\rm Gr} (n,j) \le \dim {\rm Gr} (n,k)$)  and the less general case $j+k\le n-2$ (or $\dim {\rm Gr} (n,j)) < \dim {\rm Gr} (n,k)$).
The following  statement for $j+k\le n-2$ is new. Close results for $j=0$ are due to
Kurusa \cite [Theorem 3.1]{Ku94}  and the author \cite [Theorem 5.2]{Ru21}.

\begin{theorem} \label {thrt4}  Let  $0\le j<k \le n-1$, $j+k\le n-2$.
 Then  (\ref{kvv1})  holds for all  functions  $f\in C({\rm Gr} (n,j))$  satisfying
\[A\equiv\sup\limits_{ \t \in {\rm Gr} (n,j)} |\t|^\lam |f(\t)|<\infty \quad \text {for some} \quad\lam > k-j;\]
cf. (\ref{rges}).
\end{theorem}
\begin{proof}  Let us
fix any $(k+j+1)$-plane $T$ in $\rn$ at a distance $t>r$ from the origin and let $\rho_T \in SO(n)$ be a rotation satisfying
\[
\rho_T : \, \bbr^{k+j+1} +te_n \mapsto T, \qquad \bbr^{k+j+1} = \bbr e_1 \oplus \cdots \oplus \,\bbr e_{k+j+1}.\]
Suppose that $\z\subset T$ and denote
 $\z_1=\rho_T^{-1} \z \subset \bbr^{k+j+1} +te_n$, $\; f_T (\t)=f(\rho_T \t)$.
By the assumption, $(R_{A} f)(\z)=0$ for all $\z\subset T$, and therefore
$(R_{A} f)(\rho_T \z_1) =(R_{A} f_T)(\z_1)=0$ for all $k$-planes $\z_1$ in $ \bbr^{k+j+1} +te_n$. Every such plane has the form $\z_1=\z_2 +te_n$, where $\z_2$ is a $k$-plane in $\bbr^{k+j+1}$. Hence, for any
$\z_2$  in $ \bbr^{k+j+1}$,
\be\label{kyheu}
0=(R_{A} f_T)(\z_2 +te_n)=\intl_{\z_2} f_T (\t +te_n)\, d\t= (R_{A} \tilde f)(\z_2).\ee
The latter  is the Radon $j$-to-$k$
 transform of the function $\tilde f (\t)= f_T (\t +te_n)$ on the affine Grassmannian ${\rm Gr} (n_1,j)$,  $n_1=j+k+1$.
By the assumption,
\[
  |\t|^\lam |\tilde f (\t)|=  |\t|^\lam |f(\rho_T (\t +te_n))|\le |\rho_T (\t +te_n)|^\lam |f(\rho_T (\t +te_n))| \le A.\]
    Hence, by Theorem \ref{Arther1}, (\ref{kyheu}) implies
$\tilde f (\t)= f_T (\t +te_n)=0$  for all  $\t\subset \bbr^{j+k+1}$ with $|\t|\neq 0$.
 Because $T$ and $t>r$ are arbitrary, the result follows.
\end{proof}

The support theorem for the dual Radon transform $R^*_A \vp$ can be obtained from Theorem  \ref{thrt4} by making use
of  the Kelvin-type mapping (\ref{Kelvin}) and the equality $R_A^{*} \vp= U\tilde R_A V\vp$; see (\ref{sform1}).
\begin{theorem} \label {Dthrt4}  Let  $0\le j<k \le n-1$, $j+k\ge n$,
 $\t \in {\rm Gr} (n,j)$,  $\z \in {\rm Gr} (n,k)$, $r>0$.
 Then the implication
 \be\label{Dkvv1}
(R^*_A \vp)(\t)=0\; \forall \, |\t|\in (0,r)\; \Longrightarrow \; \vp(\z)=0\; \forall \, |\z| \in (0,r)\ee
  holds for all  functions  $\vp\in C(\tilde {\rm Gr} (n,k))$  satisfying
\be\label {otaBa}
B\equiv  \sup\limits_{ \z \in \tilde {\rm Gr} (n,k)} |\z|^\del |\vp(\z)|<\infty \quad \text {for some} \quad\del < n- k;\ee
  cf. (\ref{rgesB}).
\end{theorem}
\begin{proof} We have the following implications:
\bea
(R_A^{*} \vp)(\t)=0  \;\forall \, |\t| \in (0,r) \quad  &\Longrightarrow& \quad (\tilde R_A V\vp)(\tilde \t)=0\; \forall \, |\tilde \t| > 1/r\nonumber\\
 \Longrightarrow \; (V\vp)(\tilde \z)=0\; \forall \, |\tilde \z| > 1/r \quad &\Longrightarrow& \quad \vp(\z)=0\; \forall \, |\z|\in (0,r).\nonumber\eea
The first and the third implications are obvious. The second implication holds by Theorem \ref{thrt4}  with $j$ replaced by $n-k-1$, $k$ by $n-j-1$, and $\lam$ by $n-j-\del$.
\end{proof}

It is interesting to note that  functions $\vp\in C(\tilde {\rm Gr} (n,k))$  satisfying (\ref {otaBa}) with $0<\del < n-k$ may tend to infinity as $|\z|\to 0$.

\subsection{From Grassmannians to Chords} \label{mens}

The chord  transforms in this subsection are higher-rank generalizations of the interior hyperplane Radon transforms corresponding to $j=0$, $k=n-1$. The latter were studied
by a number of authors (e.g.,  Davison  \cite {Dav},  Louis  \cite {Lou},   Maass  \cite {Maa},   Natterer \cite {Na1}, Yuan Xu   \cite {Xu}, to mention a few) in the framework of the $L^2$
theory for the corresponding singular value decompositions (SVD); see also \cite[pp. 236--245,  278]{Ru15}.
Similar SVD theory for  $j> 0$ or/and $k< n-1$, when representation theory comes into play, seems to be unknown. However, some results can be derived from those for $R_A$ and $R^*_A$ in the previous subsections.  These results might be  of interest on their own. They will also be used in Sections 4 and 5  in our study of  the Beltrami-Klein model of the hyperbolic space.

 We denote
\be\label {cte}
\Gam_B (n,d)=\{ \t \in {\rm Gr} (n,d):\, |\t| <1 \}. \ee
 The  elements of  $\Gam_B (n,d)$ can be identified with $d$-dimensional chords  of  $B_n$.
  Given a function $f$ on  $\Gam_B (n,d)$ we  define its extension to all $ \t \in {\rm Gr} (n,d)$ by setting $\tilde f (\t) =f (\t)$ if  $|\t|<1$ and  $\tilde f (\t) =0$  if  $|\t|\ge 1$.
  As in Subsection \ref{deta}, every  $\t \in \Gam_B (n,d)$ can be
   parameterized by $\t=\t (\xi, u)$,  $\xi \in G_{n,d}$, $ u \in \xi^\perp \cap B_n$, and we have
\be\label {funcE} \intl_{\Gam_B (n,d)} \!\!f (\t)\,  d\t\equiv \intl_{{\rm Gr} (n,d)} \!\! \tilde f (\t)\,  d\t =\intl_{G_{n,d}} d\xi \intl_{\xi^\perp\cap B_n} f(\xi, u)\, du.\ee

Let $ 0 \le j<k\le n-1$.    Given a function $f$ on the space $\Gam_B (n,j)$ of $j$-chords and a  function $\vp$ on the space $\Gam_B (n,k)$ of $k$-chords,
we introduce a dual pair of the {\it chord Radon  transforms}
\be\label {lchyasz} (R_B f)(\z)  = \intl_{\t \subset \zeta} f(\t)\, d_{\z} \t\equiv (\rho_B R_A \tilde f)(\z), \quad  \z \in \Gam_B (n,k), \ee
\be\label {lch5yasz} (R^*_B \vp)(\t)  = \intl_{\z \supset \t} \vp(\z) d_{\t} \z\equiv (\rho_B R^*_A \tilde \vp)(\t),\quad  \t \in \Gam_B (n,j), \ee
where $R_A$ and $R^*_A$ are affine Radon transforms  (\ref{lasz}) and (\ref{laszcv}),  $\rho_B$ stands for the restriction map to  planes meeting $B_n$. Clearly, if $\z \cap B_n \neq \emptyset$ and $\z\supset \t$, then
$\t \cap B_n \neq \emptyset$, and therefore
\be\label {funcE1}
(R^*_B \vp)(\t) \equiv (\rho_B R^*_A \tilde \vp)(\t)= (\rho_B R^*_A  \vp)(\t).\ee
 The duality relation (\ref{wbcdaff}) applied to  $\tilde f$ and $ \tilde\vp$ gives
\be\label {ctew}
 \intl_{\Gam_B (n,k)} (R_B f)(\z) \,\vp (\z)\, d \z=\intl_{\Gam_B (n,j)} f (\t)\, (R^*_B \vp)(\t)\, d\t.\ee
 In particular, setting $\vp=1$, we obtain
\be \label {ally} \intl_{\Gam_B (n,k)} \! (R_B f)(\z)\,  d \z=\intl_{\Gam_B (n,j)}  \! f(\t)\, d\t. \ee
  Further,   Proposition \ref {nsfo5} yields the following statement for radial functions.

\begin{proposition} \label{iant} Given $\t \in \Gam_B (n,j)$ and $\z \in \Gam_B (n,k)$,
let $ r=|\t|, \, s=|\z|$. If $f(\t)=f_0 (r)$ and $\vp(\z)= \vp_0
(s)$, then
\be\label {nfty} (R_B f)(\z)= \sig_{k-j-1}\intl_s^1 f_0(r) (r^2 -s^2)^{(k-j)/2
-1} r dr, \ee
\be \label {nfty1} (R^*_B \vp)(\t)= \frac{\sig_{k-j-1} \,
\sig_{n-k-1}}{\sig_{n-j-1} \, r^{n-j-2}}\intl_0^r \vp_0 (s) (r^2
-s^2)^{(k-j)/2 -1} s^{n-k-1} ds, \ee provided that  the corresponding
integrals exist in the Lebesgue sense.
\end{proposition}

It is natural that  equalities for the left-sided integrals in Propositions \ref{nsfo5} and \ref{iant} coincide.

Using tables of integrals, Proposition \ref{iant} enables one to compute  $R_B f$ and $R^*_B \vp$ for a number of
elementary functions. For example,
\bea
\label {nding2a}    \frac{(1\! - \! |\t|^2)^{\a/2-1}}   {|\t|^{\a+k-j}} \quad  &\stackrel {R_B} {\longrightarrow }&
\quad  \lam_1\, \frac{(1 \! - \! |\z|^2)^{(\a+k-j)/2 -1}}{|\z|^\a};\qquad \\
\label {nding3}
  (a^2\! - \! |\t|^2)_+^{\a/2-1} \quad  &\stackrel {R_B} {\longrightarrow }&  \quad  \lam_1\,  (a^2 \! - \! |\z|^2)_+^{(\a+k
-j)/2 -1},\; a\in (0,1];\\
\label {nding1}
 |\z|^{\a +k-n} \quad  &\stackrel {R^*_B} {\longrightarrow }&  \quad   \lam_2 \,|\t|^{\a +k-n};\\
\label {nding}
(1-|\z|^2)^{(j-n)/2} \quad  &\stackrel {R^*_B} {\longrightarrow }&  \quad   (1-|\t|^2)^{(k-n)/2}. \eea
Here $Re \,\a> 0$,
\[
\lam_1= \frac{\pi^{(k-j)/2} \Gam (\a/2)}{\Gam ((\a+k -j)/2)},  \qquad   \lam_2=\frac{ \Gam (\a/2) \, \Gam ((n-j)/2)}{\Gam ((\a+k -j)/2) \, \Gam ((n-k)/2)}.\]
Formulas  (\ref{nding3}), (\ref{nding1}) agree with
 \cite [formulas  (2.13), (2.14)]{Ru04d}.
 To prove (\ref{nding2a}), for $f(\t)=   |\t|^{j-k-\a}  (1\! - \! |\t|^2)^{\a/2-1}$   we have
 \bea
 (R_B f)(\z)&=& \sig_{k-j-1}\intl_s^1 \frac{ (1\! - \! r^2)^{\a/2-1} (r^2 \!- \!s^2)^{(k-j)/2-1}}{r^{\a+k-j}}\,  r dr\nonumber\\
 &=& \frac{\sig_{k-j-1}}{2} \frac{ (1 \! - \! s^2)^{(\a+k-j)/2 -1}}{s^{\a}}\, B\left (\frac {\a}{2}, \frac{k-j}{2}\right )\nonumber\eea
(use, e.g., \cite [formula 2.2.6 (2)]{PBM1}), which gives the result.
 To prove (\ref{nding}), for
$\vp(\z)= (1-|\z|^2)^{(j-n)/2}$ we similarly have
\[(R^*_B \vp)(\t)\!=\! \frac{\sig_{k-j-1} \,
\sig_{n-k-1}}{\sig_{n-j-1} \, r^{n-j-2}}
\intl_0^r \frac{(r^2\!-\!s^2)^{(k-j)/2 -1} s^{n-k-1}}{(1-s^2)^{(n-j)/2} } \,ds \!=\! (1-r^2)^{(k-n)/2}.\]

Combining (\ref{nding2a})-(\ref{nding})  with duality  (\ref{ctew}), we obtain
\be \label {waff}
\intl_{\Gam_B (n,k)} \!\!(R_B f)(\z)\,  |\z|^{\a +k-n}\, d\z \stackrel { (\ref{nding1})} {=}\lam_2 \intl_{\Gam_B (n,j)}\!\!
 f(\t) \,|\t|^{\a +k-n} d\t,\ee
\be \label {uwaff}
\intl_{\Gam_B (n,k)}\! \!\! \!(R_B f)(\z)\,  (1\!-\!|\z|^2)^{(j-n)/2}\, d\z   \stackrel {(\ref{nding})}{=}\!\!\intl_{\Gam_B (n,j)}\!\! \!\!
 f(\t) \,(1\!-\!|\t|^2)^{(k-n)/2} d\t.\ee
\be \label {vwaff1}\intl_{\Gam_B (n,j)} \!\!\!\!(R^*_B \vp)(\t) (a^2\! - \! |\t|^2)_+^{\a/2-1} d\t \! \stackrel {(\ref{nding3})}{=}\! \lam_1 \!\!\!\intl_{\Gam_B (n,k)}\!\!\! \!\!\vp (\z) (a^2 \! - \! |\z|^2)_+^{(\a+k
-j)/2 -1} d \z,\ee
\be \label {vwaff1D}\intl_{\Gam_B (n,j)} \!\!(R^*_B \vp)(\t) \frac{(1\! - \! |\t|^2)^{\a/2-1}}   {|\t|^{\a+k-j}}\, d\t \! \stackrel {(\ref{nding2a})}{=}\!
\lam_1 \!\!\intl_{\Gam_B (n,k)}\!\!\!\!\vp (\z)\, \frac{(1 \! - \! |\z|^2)^{(\a+k-j)/2 -1}}{|\z|^\a}\,  d \z,\ee
Many facts for $R_A $ and $R^*_A$ yield the corresponding results for  $R_B $ and $R^*_B$.

\begin{lemma}\label {ctally}  Let $ 0 \le j<k\le n-1$. If
\be\label{ensi} \intl_{\{\t \in \Gam_B (n,j):\,  a<|\t| <1\}} \!\!\!\!\!\! f(\t) \, d\t <\infty \quad \text{for all} \quad a\in (0,1),\ee
then $(R_B f)(\z)$ is finite for almost all $\z\in \Gam_B (n,k)$.
 \end{lemma}
\begin{proof} By (\ref{lchyasz}), it suffices to show that $(R_A \tilde f)(\z)$ is finite a.e. on  $\Gam_B (n,k)$. To this end, we make use of Lemma \ref{lrrb1}, according to which
\[ \intl_{|\t|>a} \frac{|\tilde f(\t)|}{|\t|^{n-k}}\, d\t \le \frac{1}{a^{n-k}}\intl_{a<|\t| <1} \!\! \!|f(\t)| \, d\t <\infty.\]
Hence, we are done.
\end {proof}

The following statement follows from (\ref{nding3}).

\begin{proposition}\label {taffy}  Let $ 0 \le j<k\le n-1$. If
\be\label {rweges} \tilde A \equiv {\rm ess} \!\!\!\!\sup\limits_{ \t \in \Gam_B (n,j)} (1\! - \! |\t|^2)^{1-\a/2} |f(\t)|<\infty \quad \text{for some $\a > 0$},\ee
 then  $(R_B f)(\z)$ is finite for  all $\z \in \Gam_B (n,k)$. The condition  $\a > 0 $  is sharp.
 \end{proposition}
 \begin{proof} Let $f_1(\t)= (1\! - \! |\t|^2)^{\a/2 -1}$. Then, by (\ref{nding3}) with $a=1$,
\[
 |(R_B f)(\z)| \le \tilde A (R_B f_1)(\z)= \tilde A \lam_1\,  (1 \! - \! |\z|^2)_+^{(\a+k-j)/2 -1},\]
 which gives the result.
\end {proof}

\begin{proposition}\label {ctally1}  Let $ 0 \le j<k\le n-1$, $a\in (0,1]$. If  $\vp$  is integrable on the set
$\{ \z \in \Gam_B (n,k): \; |\z|<a\}$,
 then $(R^*_B \vp)(\t)$ is integrable on the set
$\{ \t \in \Gam_B (n,j): \; |\t|<a\}$.
 \end{proposition}
\begin{proof}  Using (\ref{vwaff1}) with $\a=2$, we obtain
\[ \intl_{|\t|<a} \!\!|(R^*_B \vp)(\t)| \, d\t \le \lam_1 a^{k-j}
\!\intl_{|\z|<a}\!\!\! |\vp (\z)| \, d \z<\infty,\]
as desired.
\end {proof}

Further, Proposition \ref{ther1} implies the following result.

\begin{theorem}\label {tally}  Let $ 0 \le j<k\le n-1$.

\vskip 0.1 truecm

\noindent {\rm (i)} If  $j+k\le n-1$, then $R_B$ is injective on $L^1 (\Gam_B (n,j))$.


\vskip 0.1 truecm

\noindent {\rm (ii)} If  $j+k\ge n-1$, then $R^*_B$ is injective on $L^1 (\Gam_B (n,k))$.
\end{theorem}
 \begin{proof}  {\rm (i)}  Suppose $R_B f=0$ a.e. for some $f\in L^1 (\Gam_B (n,k))$. Then, by (\ref{lchyasz}), $(R_A \tilde f)(\z)=0$  for almost all $\z \in \Gam_B (n,k)$,
 and therefore  $(R_A \tilde f)(\z)=0$  for almost all $\z \in  {\rm Gr} (n,k)$. Because the map $f \mapsto \tilde f$ obviously acts from  $L^1 (\Gam_B (n,j))$ to  $L^1_{n-k} ({\rm Gr} (n,j))$,
 by Proposition \ref{ther1} we conclude that $\tilde f(\t)=0$  for almost all $\t \in {\rm Gr} (n,j)$. Hence $ f(\t)=0$  for almost all $\t \in \Gam_B (n,j)$.


{\rm (ii)} We make use of the Kelvin-type mapping
(\ref {Kelvin}) and Proposition \ref{ans}. Suppose that $(R^*_B \vp)(\t)=0$ for almost all $\t \in \Gam_B (n,j)$, where  $\vp \in L^1 (\Gam_B (n,k))$. Then, by  (\ref{lch5yasz}) and  (\ref{sform1}),
$ (\tilde R_A V\tilde \vp)(\tilde\t)=0$ for almost all $\tilde\t\in \tilde {\rm Gr} (n,k_1)$, $k_1=n-j-1$. Hence, by  Proposition \ref{ther1}, $V\tilde \vp =0$ a.e. on  ${\rm Gr} (n,j_1)$, $j_1=n-k-1$, provided that
 $V\tilde \vp \in L^1_{n-k_1} ({\rm Gr} (n,j_1))$, $j_1 +k_1 \le n-1$. The latter is equivalent  to  $j+k\ge n-1$. To complete the proof, it remains to note that by (\ref{adon})
the equality $V\tilde \vp =0$ yields   $\vp =0$, and the relation  $V\tilde \vp \in L^1_{n-k_1} ({\rm Gr} (n,j_1))$ is implied by $\vp \in L^1 (\Gam_B (n,j))$, owing to (\ref{ildes}).

  \end {proof}

The following support theorem for $R^*_B$  is a consequence of (\ref{funcE1}) and the Theorem \ref {Dthrt4}  for $R^*_A$.

\begin{theorem} \label {Ddt4}  Let  $0\le j<k \le n-1$, $j+k\ge n$,
 $\t \in  \Gam_B (n,j)$,  $\z \in \Gam_B(n,k)$, $r \in (0,1)$.
 Then the implication
 \be\label{Ddkvv1}
(R^*_B \vp)(\t)=0\; \forall \, |\t| \in (0,r)\; \Longrightarrow \; \vp(\z)=0\; \forall \, |\z|  \in (0,r)\ee
  holds for all   functions  $\vp$  which are  continuous on the set $\{ \z: 0<|\z|<r\}$ and satisfy
  \be\label {otaB}
B\equiv \sup\limits_{0<|\z|<r} |\z|^\del |\vp(\z)|<\infty \quad \text {for some} \quad\del < n- k.\ee
\end{theorem}


\section{Transition to the Elliptic Case and  Compact Grassmannians}\label {plic}

On many occasions, it is convenient  to associate $d$-dimensional affine planes in $\rn$ with  $(d+1)$-dimensional linear subspaces of $\bbr^{n+1}$ and thus
reduce the consideration of affine Grassmannian bundles to compact Grassmannians.
 Specifically, following \cite [Section 3]{Ru04d},
to each $d$-plane $\t$ in $\rn$ we assign a
$(d+1)$-dimensional linear subspace $\tau_0$ in $\rnn$ containing
 the ``lifted'' plane $\t +e_{n+1}$. This leads to a
map
\be \label {nra} {\rm Gr} (n,d) \ni \t=(\xi,u) \xrightarrow {\quad \mu \quad }
\tau_0 =\mu (\t)= \xi \oplus \bbr u_0 \in G_{n+1, d+1},\ee
 \be u_0=\frac{u+e_{n+1}}{|u+e_{n+1}|}=\frac{u+e_{n+1}}{\sqrt
{1+|u|^2}} \in S^n. \ee
The map $\mu$ is not one-to-one, but this fact does not cause any trouble; cf. \cite [Remark  3.1] {Ru04d}. Abusing notation, for the sake of simplicity we write
\be \label {nra1} |\tau_0|=d(\tilde o, \tau_0)\ee for  the geodesic
distance (on $S^n$) between the north pole $\tilde o \!=\!(0, \dots, 0,1)$ and the
$d$-dimensional totally geodesic submanifold $S^n \cap \tau_0$. Then
\be \label {nra2} |\t|=|u|=\tan |\tau_0|.\ee
The set
\be\label {mrrtr}                   \{ \tau_0 \in  G_{n+1, d+1}: |\tau_0| <\om\}, \qquad 0\le \om < \pi/2,\ee
 can be thought of as an open ball in  $G_{n+1, d+1}$ of ``radius'' $\om $ with  ``center'' composed by $(d+1)$-subspaces containing the last coordinate axis.
 This center can be identified with the lower-dimensional Grassmannian $G_{n, d}$.

\begin{lemma}\label {anra} {\rm (\cite [Lemma 3.2] {Ru04d})}  Let $\t \in {\rm Gr} (n,d)$, $ \t_0=\mu(\t) \in G_{n+1,d+1}$. Then
\be \label {nra3X} \intl_{{\rm Gr} (n,d)}  f(\t) \,
d\t=\frac{\sig_{n}}{\sig_{d}}\intl_{G_{n+1, d+1}} \frac{(f\circ \mu^{-1})(\t_0) \,
d\t_0}{(\cos \, |\tau_0|)^{n+1}}, \ee
\be  \label {nra3} \intl_{{\rm Gr} (n,d)} \frac{f(\t) \,d\t}{(1+|\t|^2)^{(n+1)/2}}=\frac{\sig_{n}}{\sig_{d}}
\intl_{G_{n+1,d+1}} (f\circ \mu^{-1})(\t_0) \, d\t_0, \ee
where $d\t_0$ stands for the Haar probability measure on $G_{n+1,d+1}$.
\end{lemma}

The correspondence (\ref{nra}) extends to  Radon
transforms. Let $G_{n+1, j+1}$ and
$G_{n+1, k+1}$ be a pair of  compact
  Grassmann manifolds, $j<k$. For a function $F$ on $G_{n+1,
  j+1}$, we consider the Radon (or Funk-Radon) transform
  \be \label {trou}(R_0 F)(\z_0)=\intl_{\t_0 \subset \z_0} F(\tau_0)\, d_{\z_0} \tau_0, \qquad \z_0 \in G_{n+1,
  k+1}, \ee
  which integrates $F$ over  all
  $(j+1)$-dimensional subspaces $\tau_0$ in $\z_0$ with respect to the corresponding
  probability measure.
The dual Radon transform $(R_0^* \Phi)(\t_0)$ of a
 function $\Phi $ on $G_{n+1, k+1}$ integrates $\Phi$ over the set  of all
  $(k+1)$-dimensional subspaces  $\z_0$ containing $\t_0$, namely,
 \be \label {trou1}(R_0^* \Phi)(\t_0)=\intl_{\z_0 \supset \t_0} \Phi (\z_0)\, d_{\t_0} \z_0, \qquad \t_0 \in G_{n+1,
 j+1}. \ee
See, e.g, \cite [Subsection 3.2] {Ru04d} for precise definitions.

Operators $R_0$ and   $R_0^*$ are expressed one through another by orthogonality. Specifically, given $\t_0 \in G_{n+1,  j+1}$ and $\z_0 \in G_{n+1,  k+1}$, we set
\[ t_0=\t_0^\perp \in G_{n+1, n-j}, \qquad z_0=\z_0^\perp \in G_{n+1, n-k}.\]
For functions $F(\t_0)$ on $G_{n+1, j+1}$ and $\Phi(\z_0)$ on $G_{n+1, k+1}$, we denote
\[ F^\perp (t_0)=F(t_0^\perp), \qquad \Phi^\perp (z_0)=\Phi(z_0^\perp).\]
Then
 \be\label {xte} (R_0\Phi^\perp )(t_0)=(R_0^* \Phi)(\t_0), \qquad (R_0F)(\z_0)=(R_0^*F^\perp)(z_0).\ee

To establish connection between Radon transforms on affine and compact Grassmannians, we set
\[ f\equiv f (\t), \quad \vp\equiv \vp (\z), \quad F\equiv F (\t_0), \quad \Phi\equiv \Phi (\z_0),\]
\[
\t \in {\rm Gr} (n,j), \quad \z\in {\rm Gr} (n,k), \quad \t_0 \in G_{n+1, j+1}, \quad \z_0 \in G_{n+1, k+1}.\]
Denote
 \bea \label {trou2}(M_0 f)(\t_0)&=& \frac{\sig_k}{\sig_j}\, (\cos |\t_0|)^{-k-1} (f\circ \mu^{-1})(\t_0), \\
  \label {trou3} (N_0 \Phi)(\z)&=& (1+|\z|^2)^{-(j+1)/2} (\Phi \circ \mu)(\z), \\
 \label {trou4} (P_0 \vp)(\z_0)&=&(\cos |\z_0|)^{j-n} (\vp\circ \mu^{-1})(\z_0), \\
 \ \label {trou5} (Q_0 F)(\t)&=& (1+|\t|^2)^{(k-n)/2} (F \circ \mu)(\t). \eea
Then the inverse maps have the form
\bea \label {Itrou2}(M_0^{-1} F)(\t)&=& \frac{\sig_j}{\sig_k}\, (1+ |\t|^2)^{-(k+1)/2} (F\circ \mu)(\t), \\
\label {Itrou3} (N_0^{-1} \vp)(\z_0)&=& (\cos |\z_0|)^{-j-1} (\vp \circ \mu^{-1})(\z_0), \\
\label {Itrou4} (P_0^{-1} \Phi)(\z)&=& (1+ |\z|^2)^{(j-n)/2} (\Phi \circ \mu)(\z), \\
 \label {Itrou5} (Q_0^{-1} f)(\t_0)&=&  (\cos |\t_0|)^{k-n}  (f\circ \mu^{-1})(\t_0). \eea

 \begin{proposition} \label {trou6} {\rm (cf. \cite [Theorems  3.4, 3.5] {Ru04d})} We have
 \be \label {trou7}
R_A f = N_0  R_0 M_0 f, \qquad R_{A}^* \vp =Q_0 R_0^* P_0 \vp,\ee
 \be \label {trou77}
R_0 F=N_0^{-1} R_A M_0^{-1}F,   \qquad R_0^*\Phi= Q_0^{-1} R_{A}^* P_0^{-1} \Phi,\ee
provided that  integrals in these equalities exist in the Lebesgue sense.
\end{proposition}


A function $F(\t_0)$ on $G_{n+1, j+1}$ is called {\it zonal} if it is invariant under rotations about the last coordinate axis. Every such function is represented as a single-variable function of the argument $|\t_0|$.

\begin{lemma} \label{huysg} Let $0\le j<k \le n-1$. For $\t_0 \in G_{n+1, j+1}$ and  $\z_0 \in G_{n+1, k+1}$, suppose that
 $F(\t_0)=F_1 (\cos |\t_0|)$,  $\Phi(\z_0)= \Phi_1  (\sin |\z_0|)$.
Then
\[(R_0 F)(\z_0)= F_2 (\cos |\z_0|), \qquad (R_0^* \Phi)(\t_0)=  \Phi_2  (\sin |\t_0|),\]
where
\be \label{hbesg} F_2(s)= \frac{\sig_j \sig_{k-j-1}}{\sig_k\, s^{k-1}} \intl_0^s  F_1(t) (s^2 -t^2)^{(k-j)/2 -1} t^j \,dt, \ee
\be \label{hbesg1} \Phi_2(r)= \frac{\sig_{k-j-1} \,
\sig_{n-k-1}}{\sig_{n-j-1} \, r^{n-j-2}}\intl_0^r \Phi_1 (s) (r^2
-s^2)^{(k-j)/2 -1} s^{n-k-1} ds, \ee
provided that the corresponding
integrals exist in the Lebesgue sense.
\end{lemma}
\begin{proof} We write the first formula in (\ref{trou7}) as $R_0 F=N_0^{-1} R_A M_0^{-1} F$. Set for a while $F(\t_0)=\tilde F  (\tan|\t_0|)$. Then
$
(M_0^{-1} F)(\t)=c\, (1+ |\t|^2)^{-(k+1)/2} \tilde F (|\t|)$,  $c= \sig_j / \sig_k$, and (\ref{besg}) yields
\[(R_0 F)(\z_0)=\frac{c\,\sig_{k-j-1}}{ (\cos |\z_0|)^{j+1}} \!\!\intl_{\tan |\z_0|}^\infty \!\!\!  (r^2 \!-\!\tan^2 |\z_0|)^{(k-j)/2 -1}   (1\!+\! r^2)^{-(k+1)/2} \tilde F (r) \, r \,dr. \]
Changing variables $1+r^2 =1/t^2$, $s=\cos |\z_0|$, and using simple trigonometry, we obtain the result.
The proof of (\ref{hbesg1})  relies on the second formula in (\ref{trou7}), which gives $R^*_0 \Phi=Q_0^{-1} R^*_A P_0^{-1} \Phi$. Specifically, if $\Phi(\z_0)= \tilde \Phi  (\tan |\z_0|)$, then (\ref{besg1}) yields
\bea
(R^*_0 \Phi)(\t_0)&=&\frac{\sig_{k-j-1} \,
\sig_{n-k-1}}{\sig_{n-j-1}} \,\frac{(\cos |\t_0|)^{k-n}}{ (\tan |\t_0|)^{n-j-2}}\nonumber\\
&\times&\intl_0^{\tan |\t_0|} (1+t^2)^{(j-n)/2} \tilde \Phi (t) (\tan^2 |\t_0| -t^2)^{(k-j)/2 -1} t^{n-k-1} dt, \nonumber\eea
Changing variables $1+t^2 =1/(1-s^2)$, $r=\sin|\t_0|$, we arrive to (\ref{hbesg1}).
 \end{proof}

Equalities (\ref{hbesg}) and (\ref{hbesg1}) agree with known formulas in \cite[Lemma 2.4]{Ru02b}   for $j=0$.

The Funk-Radon transforms $R_0 F$ and $R^*_0 \Phi$ are well-defined for any integrable functions $F$ and $\Phi$; see \cite [Corollary 4.4]{GR}.
Below we present simple sufficient conditions which  take into account the one-sided structure of $R_0 F$ and $R^*_0 \Phi$.
Denote
\be\label {tter} \tilde G_{n+1, j+1}=\{  \t_0\in G_{n+1, j+1}: 0< |\t_0|<\pi/2\},\ee
\be\label {Dtter} \tilde G_{n+1, k+1}=\{  \z_0\in G_{n+1, k+1}: 0< |\z_0|<\pi/2\}.\ee


\begin{lemma}\label {abEba}  Let $0\le j<k \le n-1$. If
\be\label {rgeszx} A_0\equiv  {\rm ess} \!\!\!\!\sup\limits_{\t_0\in  G_{n+1, j+1}}  (\cos |\t_0|)^{p} |F(\t_0)|<\infty \quad \text{for some $p< j +1$},\ee
 then   $(R_0 F)(\z_0)$ is finite for all $\z_0\in \tilde G_{n+1, k+1}$.
If
\be\label {Argeszx} B_0\equiv  {\rm ess} \!\!\!\!\sup\limits_{\z_0\in G_{n+1, k+1}}  (\sin |\z_0|)^{q} |\Phi (\z_0)|<\infty \quad \text{for some $ q< n-k$},\ee
 then   $(R^*_0 \Phi)(\t_0)$ is finite for all $\t_0\in \tilde G_{n+1, j+1}$.
 The restrictions $p <j+1$ and $q <n-k$ are sharp.
\end{lemma}
\begin{proof}
  The first statement is derived from Proposition \ref {Arges} and the first equality in (\ref{trou77}). It suffices to show that
 \[
  A\equiv {\rm ess} \!\!\!\!\sup\limits_{ \t \in  {\rm Gr} (n,j)} |\t|^\lam |(M_0^{-1} F)(\t)|<\infty\quad \text{for some $\lam > k-j$}.\]
  Indeed, by (\ref{nra2}),
\bea
 |\t|^\lam |(M_0^{-1} F)(\t)|&=& |\t|^\lam (1+|\t|^2)^{-(k+1)/2}|(F\circ \mu)(\t)\nonumber\\
 &=&(\sin |\t_0|)^\lam (\cos |\t_0|)^{k+1-\lam} |F(\t_0)| \le (\cos |\t_0|)^{p} |F(\t_0)|, \nonumber\eea
where $p=k+1-\lam  <j+1$. This gives the result for $R_0 F$.
 The result for $R^*_0 \Phi$ similarly follows from Proposition \ref {lrrbx} (with $\del =q$) and the second equality in (\ref{trou77}).
 \end{proof}


\begin{theorem} \label {Ntrou6} Let $0\le j<k \le n-1$, The operator $R_0$ is injective on $L^1( G_{n+1, j+1})$ if and only if  $j+k\le n-1$.
The dual transform $R_0^*$ is injective on $L^1( G_{n+1, k+1})$ if and only if  $j+k\ge n-1$.
\end{theorem}

The second statement follows from the first one by orthogonality; cf. \cite [Section 4.3]{Ru04d}.  The ``if" part of the first
statement together with explicit inversion formulas were obtained by Grinberg and Rubin  \cite [Theorem 1.2]{GR}. This paper
contains references to the previous works, where the operator $R_0$ was studied on $C^\infty$ functions.
The  ``only if" part is well-predictable from the dimensionality  consideration because the condition $j+k\le n-1$ is equivalent to
\be\label {kiot} \dim G_{n+1,j+1}\le \dim G_{n+1,k+1}.\ee
The rigorous proof  of the ``only if" part falls into the scope of representation theory for the rotation group;
see, e.g., \cite{SW} for a brief discussion of this theory or \cite{Bo, We}  for a complete discussion.
The general idea is the following.\footnote{This explanation was kindly sent to the author by Prof. Tomoyuki Kakehi.}
  If (\ref{kiot}) fails, then there exists an irreducible representation subspace $V \subset C^\infty (G_{n+1,j+1})$ whose highest weight  does not occur in $C^\infty (G_{n+1, k+1})$ and therefore, by $SO(n)$-equivariance,  $V$ is mapped by  $ R_0$ to zero. Explicit formulas for the highest weights can be found in \cite{Str75, Str77}. Further details about group-theoretic aspects of the Radon transforms
on Grassmannians can be found in \cite{GW, Gr85, Gr86, Ka}.
The relevant harmonic analysis, which explains the role of  (\ref{kiot}) in the study of injectivity of Radon transforms on Grassmannian bundles,  was developed in \cite[Section 5]{Str86}.

The following injectivity statements, which are  similar to those
in Theorems \ref{Arther1} and \ref{rther1},  can be obtained by making use of the connection formulas
(\ref{trou77}); cf. the  reasoning in Lemma \ref {abEba}.

\begin{theorem} \label {Arcvh}  Let $ 0 \le j<k\le n-1$. If  $j+k\le n-2$, then $R_0$ is injective on functions $F$ satisfying (\ref{rgeszx}).  If  $j+k\ge n$, then $R^*_0$ is injective on functions $\Phi$ satisfying (\ref{Argeszx}).
\end{theorem}

Our next objective is the support theorem for the Funk-Radon transform on compact Grassmannians. We will be dealing with functions  satisfying (\ref{rgeszx}) and (\ref{Argeszx}) under the additional assumption that these functions are continuous on the open subsets  $\tilde G_{n+1, j+1}$ and  $\tilde G_{n+1, k+1}$; see (\ref{tter}), (\ref{Dtter}).
As in Theorem \ref {thrt4}, we distinguish the cases  $j+k=n-1$ and $j+k\le n-2$. For the second case, we have the following statement.

\begin{theorem} \label {Cthrt4} Let $0\le j<k \le n-1$, $0<\om <\pi/2$.

\vskip 0.1 truecm

\noindent {\rm (i)}  If  $j+k\le n-2$, then  the implication
\be\label{Ikvv1}
(R_0 F)(\z_0)=0\; \forall \, |\z_0| > \om\; \Longrightarrow \; F(\t_0)=0\; \forall \, |\t_0| > \om\ee
holds for any  function  $F\in C(\tilde G_{n+1, j+1})$ satisfying (\ref{rgeszx}).

\noindent {\rm (ii)}  If $j+k\ge n$, then  the implication
\be\label{Ikvv1}
(R^*_0 \Phi)(\t_0)=0\; \forall \, |\t_0| < \om\; \Longrightarrow \; \Phi(\z_0)=0\; \forall \, |\z_0| < \om\ee
holds  for any  function  $\Phi \in C(\tilde G_{n+1, k+1})$ satisfying (\ref{Argeszx}).
\end{theorem}
\begin{proof}
 {\rm (i)} If $(R_0 F)(\z_0)=0$ for all $|\z_0| > \om$, then, by
 Theorem \ref{trou6},
$(R_A M_0^{-1} F)(\z)=0$ for all $ |\z| > r=\tan \om$. Hence, by Theorem \ref {thrt4},
\be\label{swvv1} (M_0^{-1} F)(\t)=0\quad \text{\rm  for all }\quad |\t| > r,\ee
provided that $M_0^{-1} F$ satisfies (\ref{rges}),
 that is, $F$ satisfies (\ref{rgeszx}).  This completes the proof.

{\rm (ii)}
Suppose that $(R^*_0 \Phi)(\t_0)=0$ for all $ |\t_0| < \om$.  Because $R_0^*\Phi= Q_0^{-1} R_{A}^* P_0^{-1} \Phi$,  owing to (\ref{Itrou5}), we have
$(R_{A}^* P_0^{-1} \Phi)(\mu^{-1} (\t_0))=0$  for all $ |\t_0| < \om$, which is equivalent to  $(R_{A}^* P_0^{-1} \Phi)(\t)=0$  for all $ |\t| < \tan\om$. Hence, by Theorem \ref {Dthrt4},
 $(P_0^{-1} \Phi)(\z)=0$  for all $ |\z| < \tan\om$ if $P_0^{-1} \Phi$  satisfies (\ref{otaB})
  or  $\Phi$ satisfies (\ref{Argeszx}).
 \end {proof}

\begin{remark}\label {owl}
To the best of our knowledge,  an analogue of Theorem \ref {Cthrt4} for the case  $j+k= n-1$ with $j>0$ is unknown.\end{remark}

Inversion formulas for Funk-Radon transforms on compact Grassmannians can be found, e.g.,  in \cite{GGR, Gr86, GR, Ka, P1, Ru04d, Ru20, Zha1, Zh}.

\section {The Real Hyperbolic Space. Preliminaries}

Let $\bbe^{n, 1}$, $n\ge 2$, be  the  pseudo-Euclidean $(n+1)$-dimensional real vector space of points
 with the inner product
\be\label {tag 2.1-HYP}[{\bf x}, {\bf y}] = - x_1 y_1 - \ldots -x_n y_n + x_{n +1} y_{n +1}. \ee
We denote by  $  e_1, \ldots, e_{n +1}$  the coordinate unit vectors  in $\bbe^{n,1}$.
 The pseudo-orthogonal group of linear transformations preserving the bilinear form $[{\bf x},  {\bf y}]$ is denoted by $O(n,1)$;
$SO(n,1)=\{ g\in O(n,1): \det (g)=1\}$;
$G=SO_0(n,1)$ is the identity component of $SO(n,1)$.
 The elements of $G$ are called {\it hyperbolic rotations}. The unit hyperboloid
\[\hn = \{{\bf x}\in \bbe^{n,1} :
[ {\bf x}, {\bf x}] = 1, \ x_{n +1} > 0 \},\]
is considered as a model of the $n$-dimensional real hyperbolic space. We will also deal with some other models to be introduced later.
The points of $\hn$  will be denoted by the non-boldfaced letters, unlike the generic points in $\bbe^{n,1}$.

The notation $o=(0, \ldots, 0,0)$ is used for the origin  of  $\bbe^{n,1}$ or $\rn$, depending on the context.
The point ${\tilde o}=(0, \ldots, 0,1)$ serves as the origin of $\hn$ and the north pole of $S^n$, the unit sphere in $\bbr^{n+1}$.
The geodesic distance between  points $x$ and $y$ in $\hn$ is defined by $d(x,y) = \cosh^{-1}[x,y]$.
 The subgroup
\[
K= \left \{ \left[ \begin{array}{cc}
 \gam    & 0           \\
 0               & 1
 \end{array} \right ]:\, \gam \in SO(n)\right \} \]
 is the isotropy subgroup of $\tilde o$ in $G$, and therefore $\hn$ is equivariantly diffeomorphic to the quotient space $G/K$.
The Haar measure $dg$ on $G$ will be normalized by the equality
\be\label {invam} \intl_Gf(g\tilde o)\,dg=\intl_{\hn}
f(x)\, dx,\ee where $dx$ stands for the usual Riemannian measure on $\hn$; see, e.g.,  \cite [p. 367]{Ru15}, \cite [p. 23] {VK}.
This normalization agrees with the more intuitive  condition that the Haar measure of $K$  is one.

 If $d$ is an integer, $ 0 \le d \le n-1$, we set $\bbe^{n, 1} = \bbr^{n-d}\times\bbe^{d,1}$, where
\be \bbr^{n-d} = \bbr e_1 \oplus\cdots\oplus  \bbr e_{n-d},\qquad
\bbe^{d,1} = \bbr e_{n-d+1} \oplus\cdots\oplus  \bbr e_{n+1}.
\ee
This convention will be used everywhere in the sequel. In particular, $\rn = \bbr e_1 \oplus\cdots\oplus  \bbr e_{n}$ is identified with the coordinate subspace of  $\bbe^{n, 1}$.

\subsection{Generalized
Hyperbolic Coordinates}

\begin{lemma} Let $0\le d\le n-1$. Each point $x \in \hn$ can be represented in the generalized
hyperbolic coordinates as
\be\label{23} x =
\eta\cosh r+\zeta\sinh r = \left[ \begin{array}{c}
 \zeta\sinh r          \\
 \eta\cosh r
 \end{array} \right ],\ee
\[\eta\in\bbh^d \subset \bbe^{d,1},\qquad \zeta\in S^{n-d-1}\subset \bbr^{n-d}, \qquad 0 \le  r<\infty,\] so that
\be\label{24} dx=d\eta \, d\zeta \, d\nu(r), \quad
d\nu( r)=(\sinh r)^{n-d-1}(\cosh r)^d \, d r,\ee
$d\eta$ and $d\zeta$ being the Riemannian measures on $\bbh^d$ and
$S^{n-d-1}$, respectively.
\end{lemma}
\begin{proof} The equality (\ref{24}) can be found in \cite [pp. 12, 23]{VK} without proof.  For convenience of the reader, we prove it below. We write  $x \in \hn$ as
\[
x= x' +x''= \frac {x'}{|x'|}\, |x'| +  \frac {x''}{|x''|_H}\, |x''|_H,  \qquad    x'\in \bbr^{n-d},  \quad x''\in \bbe^{d,1},\]
where
\[ |x'| \!=\! \sqrt {x_1^2 + \cdots +  x^2_{n-d}}, \qquad  |x''|_H \!=\! \sqrt {-x^2_{n-d+1}- \cdots -x^2_{n} + x^2_{n+1}}.\]
Setting
 \[  \eta=\frac {x''}{|x''|_H}, \qquad \z= \frac {x'}{|x'|}, \qquad \ch \,r=|x''|_H, \qquad \sh \,r=|x'|, \]
 we obtain (\ref{23}). To prove (\ref{24}),  let
\[ \Om_{n+ 1} =\{ {\bf x}\in \bbe^{n,1} : \, - x_1^2  - \ldots -x_n^2 + x^2_{n +1}>0,\; x_{n+1}>0\}\]
be the interior of the light cone in $\bbe^{d,1}$;  $\Om_{d+ 1}=\Om_{n+ 1}\cap \bbe^{d, 1}$. Then for any sufficiently good function $F$ on $\Om_{n+ 1}$,
\[
\intl_{\Om_{n+ 1}} F({\bf x})\, d{\bf x}=\intl_{\Om_{d+ 1}}\, d{\bf x}''\intl_{ |{\bf x}'| < |{\bf x}''|_H }  F({\bf x}'+{\bf x}'' )\,  d{\bf x}''\]
or
\[\intl_0^\infty t^n dt  \intl_{\hn} F(ty)\, dy=\intl_0^\infty s^d ds \intl_0^s w^{n-d-1} dw \intl_{\bbh^d} d\eta \intl_{S^{n-d-1}} F(w \z + s\eta)\, d\z.\]
Changing variables $(s,w) \to (t\,\ch\, r,  t\,\sh \, r)$ (the Jacobian  $=t$), we obtain
\bea
\intl_0^\infty t^n dt  \intl_{\hn} F(ty)\, dy&=&\intl_0^\infty t^n dt \intl_0^\infty (\sinh r)^{n-d-1}(\cosh r)^d \, d r \nonumber\\
&\times&\intl_{\bbh^d} d\eta \intl_{S^{n-d-1}} F(t (\eta\cosh r+\zeta\sinh r))\, d\z. \nonumber\eea
Choose $F$ in the form $ F(ty)= g(t) f(y)$ with $\int_0^\infty g(t) t^n dt =1$. This gives
\bea
\intl_{\hn} \!f(y)\, dy \!\!\!&=&\!\!\!\intl_0^\infty (\sinh r)^{n-d-1}(\cosh r)^d \, d r \intl_{\bbh^d} d\eta \!\!\intl_{S^{n-d-1}}  \!\! \!f(\eta\cosh r+\zeta\sinh r) d\z\nonumber\\
\label {onu}&=&\!\!\intl^\infty_0d\nu( r)\intl_{\bbh^d}
d\eta\intl_{S^{n-d-1}} f(\eta\cosh r+\zeta\sinh r) \, d\zeta,\eea
and the result follows.
\end{proof}

\begin{remark} The integral (\ref{onu}) can be written in a more compact form

\be\label {poi}\intl_{\hn} \!f(y)\, dy=\sigma_{n-d-1}\intl^\infty_0\Big[\intl_{H_d} f(h g_{n-d, n+1} ( r) \,\tilde o )\,dh\Big] \, d\nu( r), \ee
where
\be\label {hipr} g_{n-d, n+1} (r) =
 \left[ \begin{array}{cccc}
 I_{n-d-1}     & 0             & 0        & 0              \\
 0           & \cosh r   & 0        & \sinh r    \\
 0           & 0             & I_{d}    & 0              \\
 0           &\sinh r    & 0        & \cosh r
 \end{array} \right ]
\ee
is the hyperbolic rotation in the coordinate plane $(x_{n-d}, x_{n+1})$,
\be\label {ford}
H_d= \left \{ h=\left[ \begin{array}{cc}
 \a_1    & 0           \\
 0               & \a_2
 \end{array} \right ]:\, \a_1 \in SO(n-d),  \,\a_2 \in SO_0(d,1)\right \}, \ee
\[dh=d\a_1 d \a_2.\]
Here $d\a_1$ stands for the Haar probability measure on $SO(n-d)$  and the measure  $d\a_2$ is normalized as in (\ref {invam}) with $G$ replaced by $SO_0(d,1)$, $\hn$ by $\bbh^d$, and $dx$ by the  Riemannian measure
$d\zeta$  on the hyperboloid  $\bbh^d$. Note that the choice of the hyperbolic rotation (\ref{hipr}) in (\ref{poi}) is not unique because of the $H_d$-invariance of the measure $dh$; cf. \cite [formula (2.19)]{BR1}.
\end{remark}

\subsection{$d$-Geodesics}

We denote by $\Gam_H (n,d)$, $0\le d\le n-1$,   the set of all $d$-dimensional totally geodesic submanifolds of $\hn$ ($d$-geodesics, for short).
 If $d=0$, then  $\Gam_H (n,d)$ is identified with $\hn$. Each $d$-geodesic is a cross-section of the hyperboloid $\hn$ by some $(d+1)$-dimensional plane passing through the origin $o$ of $\bbe^{n, 1}$.
The group $G=SO_0 (n,1)$ acts on $\Gam_H (n,d)$ transitively.
We consider the $d$-dimensional hyperboloid $\bbh^d=\hn \cap \bbe^{d, 1}$ as the origin (the base element) of $\Gam_H (n,d)$. The group $H_d$ in (\ref{ford}) is the
stabilizer of  $\bbh^d$ in $G$. Hence $\Gam_H (n,d)$ can be regarded as the quotient space $G/H_d$ and each $d$-geodesic
$\fr {t} \in \Gam_H (n,d)$ has the form $\fr {t} =g \bbh^d$ for some $g\in G$.


We denote by  $|| \fr {t}||$ the geodesic distance from $\fr {t}\in \Gam_H (n,d)$ to the origin $\tilde o=(0, \ldots, 0,1)$ of $\hn$. If
$\fr {t}=\a g_{n-d, n+1} (r) \bbh^d$, where $\a \in K$ and $ g_{n-d, n+1} (r)$ is the hyperbolic rotation (\ref{hipr}), then
\bea \label{aick}
 || \fr {t}||&=&d (\tilde o,\a g_{n-d, n+1} (r) \bbh^d)= d(g_{n-d, n+1} (r) \tilde o, \bbh^d)\nonumber\\
 &=& d (e_{n-d}\, \sh \,r +e_{n+1}\, \ch \,r,  \bbh^d)=r.\eea
Once the invariant measures $dg$ on $G$ and $dh$ on $H_d$ are
 fixed, we can define the invariant measure $d\fr {t} \!= \!d(gH_d)$ on
$\Gam_H (n,d)\!=\! G/H_d$ normalized by
\be\label {helg}
 \intl_G \Phi(g)dg = \intl_{G/H_d} d(gH_d)\intl_{H_d} \Phi(gh)\,dh;\ee
see \cite [p. 91]{H00}. The following lemma gives precise expression of  $d\fr {t}$ in the hyperbolic coordinates.

\begin{lemma}  If $\varphi \in L^1(\Gam_H (n,d))$, then
\be\label {jua}
 \intl_{\Gam_H (n,d)}\varphi(\fr {t})\,d\fr {t}=\sigma_{n-d-1}\intl^\infty_0 d\nu( r) \intl_K\varphi(\a g_{n-d, n+1} (r) \bbh^d)\,
 d\a, \ee
where $d\nu( r)=(\sinh r)^{n-d-1}(\cosh r)^d \, d r$, and $g_{n-d, n+1} (r)$ is the hyperbolic rotation (\ref{hipr}).
\end{lemma}
\begin{proof} This statement  is a slight modification of our previous result \cite [Lemma 1]{BR1}. We present the proof (with minor changes) for the sake of completeness.
Let $\Pi:G\mapsto G/H_d$ be the canonical projection, and let $f_1$ be
a positive function on $G$ satisfying
  \[\tilde f_1(g) \equiv \intl_{H_d} f_1(gh)dh <\infty \quad \, \forall g \in G \]
  (such a function  obviously exists).  To each $\varphi \in L^1(\Gam_H (n,d))$ we assign a function $\Phi$ on $G$
  defined by
\be
  \Phi(g) \, = \, \frac{(\varphi\circ\Pi)(g) f_1(g)}{\tilde f_1(g)}.
  \ee
  Then
\be \label{trick} \intl_{H_d}\Phi(gh)\,dh=\varphi(gH_d).\ee
The trick related to $f_1$ is
a slight modification of Helgason's idea (see \cite [p. 91]{H00}).  Thus
we have
\bea \intl_{\Gam_H (n,d)}\varphi(\fr {t})\,d\fr {t}&=& \intl_{G/H_d}
\varphi(gH_d) \, d(gH_d) \nonumber \\
&\buildrel {(\ref{trick})}\over =& \intl_{G/H_d} d(gH_d)\intl_{H_d} \Phi(gh) \, dh
\nonumber \\
& \stackrel{(\ref{helg})}{=}& \intl_G\Phi(g) \, dg =\intl_G\tilde\Phi(g) \, dg, \quad \tilde\Phi(g)=\intl_K\Phi(\a g^{-1})\,d\a.\nonumber \eea
The function $\tilde\Phi$ is  right $K$-invariant and can be
identified with a  function $\Psi$
 on $\hn$, so that $\tilde\Phi(g)=\Psi(g \tilde o)$.  Hence
\bea
 \intl_G\tilde\Phi(g) \, dg
 &\stackrel {(\ref{invam})}{=}& \intl_{\hn}\Psi(x) \, dx \nonumber \\ &\buildrel {(\ref{poi})}\over
 =&
 \sigma_{n-d-1}\intl^\infty_0
 d\nu( r)\intl_{H_d}\Psi(hg_{n-d, n+1} (r)  \tilde o) \, dh \nonumber
\eea
and we have
\bea \intl_{H_d}\Psi(hg_{n-d, n+1} (r)  \tilde o) \, dh
&=&\intl_{H_d}\tilde\Phi(hg_{n-d, n+1} (r)) \, dh \nonumber \\ &=&
\intl_{H_d} dh\,\intl_K\Phi(\a g^{-1}_{n-d, n+1} (r) h) \, d\a \nonumber \\
\label {oka}&\stackrel{\rm (\ref{trick})}{=}&  \,\intl_K\varphi(\a g^{-1}_{n-d, n+1} (r) H_d) \,
d\a.  \eea
Note that  $g^{-1}_{n-d, n+1} (r)$ can be replaced by $g_{n-d, n+1} (r)$. Indeed,  $g^{-1}_{n-d, n+1} (r)=g_{n-d, n+1} (-r)$, and therefore
\bea
g^{-1}_{n-d, n+1} (r) \tilde o &=& g_{n-d, n+1} (-r) \tilde o\nonumber\\
&=& -e_{n-d} \,\sh r +e_{n+1} \,\ch r=\a_0 (e_{n-d}\, \sh r +e_{n+1}\, \ch r)\nonumber\\
&=&\a_0 g_{n-d, n+1} (r)\nonumber\eea
for some $\a_0 \in K$, $\a_0 e_{n-d} =- e_{n-d}$. Changing variable in (\ref {oka}), we replace $g^{-1}_{n-d, n+1} (r)$  by $g_{n-d, n+1} (r)$. This gives the result.
\end{proof}

\begin{remark} The formula (\ref{jua}) was implicitly stated
by Ishikawa \cite [p. 155] {I97} as a particular case of a general
result of Schlichtkrull \cite [p. 149]{Schl}.  Our statement is
precise and its proof is self-contained.
\end{remark}

Every zonal (i.e., $K$-invariant) function $\vp (\fr {t})$ on $\Gam_H (n,d)$ is, in fact, a single-variable function of the distance $|| \fr {t}||$. The equality (\ref{jua}) yields the following result.
\begin{corollary} If  $\vp (\fr {t})\equiv \vp_0 (\ch || \fr {t}||)$, then
\bea
\label{aicke} \intl_{\Gam_H (n,d)}\!\!\!\!\varphi(\fr {t})\,d\fr {t}\!\!\!&=\!\!\!&\sigma_{n-d-1}\intl_0^\infty (\sinh\, r)^{n-d-1}(\cosh \,r)^d \, \vp_0 (\ch \,r)\, dr \qquad \\
\label{aick1a} \!\!\!&=&\!\!\!\sigma_{n-d-1}\intl_1^\infty (s^2 -1)^{(n-d)/2 -1} s^d \, \vp_0 (s)\, ds.\eea
Alternatively, if  $\vp (\fr {t})\equiv \vp_1 (\th || \fr {t}||)$, then
\be \label{aick2}
 \intl_{\Gam_H (n,d)} \varphi(\fr {t})\,d\fr {t}=\sigma_{n-d-1}\intl_0^1 \frac{s^{n-d-1}}{(1-s^2)^{(n+1)/2}} \, \vp_1 (s)\, ds.\ee
\end{corollary}
Here (\ref {aick2}) follows from  (\ref {aicke}) if we set $\th\, r =s$.

\subsection {The Beltrami-Klein Model}

The unit ball $B_n$ in $\rn$ with the relevant metric \cite [Section 7]{CFKP}
can be considered as the Beltrami-Klein model of the real $n$-dimensional hyperbolic space. The set
$\Gam_B (n,d)$ of all $d$-dimensional chords  in  $B_n$ (see Subsection \ref{mens}) forms the manifold of all $d$-geodesics in this model.
 On the other hand, every $d$-geodesic in the hyperboloid model $\hn$ is a cross-section of  $\hn$ by a $(d+1)$-dimensional plane passing through the origin $o=(0, \ldots, 0)$ of $\bbe^{n,1}$.
 To establish connection between geodesics in both models,
 we consider the $K$-equivariant map
 \be \label {mnqe}\Gam_B (n,d) \ni \t \xrightarrow {\quad \pi \quad }
\fr {t} =\hn \cup \{\t + e_{n+1}\} \in \Gam_H (n,d),\ee
where $ \{\t + e_{n+1}\}$ is the smallest linear subspace containing the shifted plane $\t + e_{n+1}$. It is clear that $\pi $ is bijective, in contrast to the similar map $\mu :   {\rm Gr} (n,d) \mapsto G_{n+1, d+1}$ from (\ref{nra}).

\begin{lemma}\label {lem1} If $f \in L^1 (\Gam_B (n,d))$, then
\be\label {kred}
\intl_{\Gam_B (n,d)} f(\t)\, d\t= \intl_{\Gam_H (n,d)} \frac{ (f\circ \pi^{-1})(\fr {t}) }{ (\ch ||\fr {t}||)^{n+1}}\, d\fr {t},\ee
where the measure $d\fr {t}$ is defined according to ( \ref{jua}); $||\fr {t}||=d (\fr {t}, \tilde o)$.
If $\Phi \in L^1 (\Gam_H (n,d))$, then
\be\label {kred1}
\intl_{\Gam_H (n,d)} \Phi (\fr {t})\, d\fr {t}=\intl_{\Gam_B (n,d)} \frac{ (\Phi\circ \pi)(\t)}{(1-|\t|^2)^{(n+1)/2}}\, d\t.\ee

\end{lemma}
\begin{proof} We set $f(\t)=f(\xi+ u)$, $\xi \in G_{n,d}$, $ u \in \xi^\perp \cap B_n$, and transform the left-hand side of (\ref{kred}).
Passing to polar coordinates in $\xi^\perp$ and changing variables, we obtain
\bea
&&\intl_{\Gam_B (n,d)} f(\t)\, d\t= \intl_{G_{n,d}} d\xi \intl_{\xi^\perp \cap B_n} f(\xi + u) \,du\nonumber\\
&& =\intl_{G_{n,d}} d\xi \intl_0^1 t^{n-d-1}dt \intl_{S^{n-1} \cap \,\xi^\perp} f(\xi +\sig t)\, d\sig\nonumber\\
&&\;  \text {\rm (set $t=\th \,r$, $ \xi =\a \bbr^d, \; \a\in K$, $\;\sig =\a \del e_{n-d}, \; \del \in SO (n-d)\,$)}\nonumber\\
&&=\sigma_{n-d-1} \intl_{K} d\a \intl_0^\infty \frac{(\th \,r)^{n-d-1}}{(\ch \,r)^2} \,dr \!\!\intl_{SO (n-d)} \!\!\! f(\a ( \bbr^d +\del e_{n-d} \th\, r))\, d \del\nonumber\\
&&=\sigma_{n-d-1} \intl_0^\infty d\nu (r)\intl_{K} \frac{ f(\a ( \bbr^d + e_{n-d}\, \th\, r))}{(\ch \,r)^{n+1}}\, d\a,\nonumber\eea
where $d\nu( r)=(\sinh r)^{n-d-1}(\cosh r)^d$. Note that
\[
f(\a ( \bbr^d + e_{n-d} \,\th\, r))=f(\a \pi^{-1}(g_{n-d, n+1} (r) \bbh^d)),\]
where $g_{n-d, n+1} (r)$ is the hyperbolic rotation (\ref{hipr}) and   $\pi^{-1}: \Gam_H (n,d) \mapsto \Gam_B (n,d)$ is the inverse of (\ref{mnqe}). Thus, owing to $K$-equivariance of $\pi$,
\[
\intl_{\Gam_B (n,d)} f(\t)\, d\t=\sigma_{n-d-1} \intl_0^\infty d\nu (r)\intl_{K} \frac{(f\circ \pi^{-1})(\a g_{n-d, n+1} (r) \bbh^d)}{(\ch \,r)^{n+1}}\, d\a.\]
Setting $\fr {t}= \a g_{n-d, n+1} (r) \bbh^d \in \Gam_H (n,d)$, and noting that $d (\fr {t}, \tilde o)\equiv ||\fr {t}|| =r$ (cf. (\ref{aick}),  (\ref{jua})), we obtain (\ref{kred}).
The latter implies (\ref{kred1}) if we note that
\be\label {tat} 1/\ch^2 \,||\fr {t}|| =1/\ch^2 \,r =1- \th^2 r= 1- |\t|^2.\ee
\end{proof}

In addition to (\ref{tat}), we also note that
 \be\label {adt2T}
\cos \, 2|\t_0|\!= \! \frac{1\!-\!\tan^2 |\t_0|}{1\!+\!\tan^2 |\t_0|}= \frac{1\!-\!|\t|^2 }{1\!+\!|\t|^2}=\frac{1\!-\!\th^2 ||\fr {t}||}{1\!+\!\th^2 ||\fr {t}||}=\! \frac {1}{\ch \, 2||\fr {t}|||},\ee
where $\fr {t}=\pi (\t)$ and $\tau_0 =\mu (\t)$; cf.  (\ref{mnqe}), (\ref{nra}). These equalities will be used in forthcoming calculations.

\section {Higher-Rank Radon Transforms in the Hyperbolic Space}

\subsection {Setting of the Problems and Discussion}  For $0\le j<k\le n-1$, let $\Gam_H (n,j)$ and $\Gam_H (n,k)$ be a pair of the corresponding spaces of totally geodesic submanifolds of $\hn$.
   The case $j=0$ corresponds to points of $\hn$. Consider  the Radon  transforms
\be\label {larz} (R_H f)(\fr {z}) = \intl_{ \fr {t} \subset \fr {z}} f(\fr {t})\, d_{\fr {z}} \fr {t}, \qquad (R^*_H \vp)(\fr {t})  = \intl_{ \fr {z} \supset \fr {t}} \vp(\fr {z})\,d_{{}_{\fr {t}}} \fr {z},\ee
\[
\fr {t} \in \Gam_H (n,j), \qquad \fr {z} \in \Gam_H (n,k).\]

Our main objectives are

$\quad$ - precise definition  of these integrals;

$\quad$ -  existence  in Lebesgue spaces;

$\quad$ -  representation in the  Beltrami–-Klein and projective models;

$\quad$ - support theorems;

$\quad$ - inversion formulas.

The precise definition of $R_H$ and $R^*_H$ can be given in the group-theoretic terms in the framework of   Helgason's double fibration scheme
\be\label {kerx}
G/L_1  \longleftarrow  G/L  \longrightarrow G/L_2,  \qquad   L=L_1 \cap L_2.\ee
 In our case, $G=SO_0 (n,1)$, $L_1=H_j$ (the stabilizer of the sub-hyperboloid $\bbh^j$), $L_2=H_k$   (the stabilizer of the sub-hyperboloid $\bbh^k$);
 see (\ref{ford}) with $d=j,k$. According to (\ref{kerx}),
\be\label {Spec}
(R_H f)(\fr {z})\equiv (R_H f)(gL_2)=\intl_{L_2/L} f (gsL_1) ds_L, \ee
\be\label {Spec1} (R^*_H \vp)(\fr {t})\equiv (R^*_H \vp)(gL_1)= \intl_{L_1/L} \vp(gtL_2) dt_L,\ee
where $ds_L$ is the $L_2$-invariant measure on $L_2/L$ and $dt_L$ is the $L_1$-invariant measure on $L_1/L$; cf.  \cite[pp. 65, 68]{H11}, \cite[p. 237]{I20}.
This definition yields the duality relation
\be \label {wbcdhyp} \intl_{\Gam_H (n,k)} (R_H f)(\fr {z}) \,\vp (\fr {z})\, d \fr {z}=\intl_{\Gam_H (n,j)} f(\t)\, (R^*_H f)(\fr {t})\, d\fr {t} \ee
 for smooth compactly supported functions $f$ and $\vp$.

The afore-mentioned group-theoretic approach is applicable to a wide variety of Radon-like transforms and paves the way to many deep results.
 However, in our  case, it remains unclear (a) what  the minimal conditions are for  $f$ and $\vp$ under which the integrals (\ref{Spec}) and (\ref{Spec1}) exist in the standard Lebesgue sense,
  (b) how to use the the above  definitions to explicit  inversion of $R_H$ and $R^*_H$, and (c) what kind of support theorems for these operators are available.
Moreover,  the definition of smooth functions and the relevant differential operators on the  coset spaces needs a substantial
Lie-theoretic work. Below we develop an alternative,  functional-analytic approach, which relies on the integration formula (\ref{jua}).

 We make use of the $K$-equivariant map (\ref{mnqe}) (with $d$ replaced by $j$ and $k$),
 which establishes the one-to-one correspondence between chords in $B_n$ and geodesics in $\hn$.
If
\[\fr {t} \in \Gam_H (n,j), \quad \t=\pi^{-1} (\fr {t})\in \Gam_B (n,j) ; \]
\[ \fr {z}\in \Gam_H (n,k), \quad \z=\pi^{-1} (\fr {z})\in \Gam_B (n,k) ,\]
 then, clearly  $\fr {t} \subset \fr {z}$ if and only if $\t \subset \z$. It follows that all properties
of the chord transforms $R_B$ and $R^*_B$
have equivalent reformulation in terms of  $R_H$ and $R^*_H$.

Everywhere below, we will  keep the notation (\ref{krtg})-(\ref{krtg2}) for the coordinate subspaces
$\bbr^j$, $\bbr^k$, $\bbr^{n-j}$, $\bbr^{n-k}$.

\subsection {The Radon Transforms $R_H$ and $R^*_H$ } We recall that formally,  $(R_H f)(\fr {z}) = \int_{ \fr {t} \subset \fr {z}} f(\fr {t})\, d_{\fr {z}} \fr {t}$.
To obtain explicit analytic expression for this integral, let first   $d m (\fr {t})$ be the canonical measure on the space $\Gam_H (k,j)$ of $j$-geodesics in the
hyperboloid $\bbh^k$.   Then $d_{\fr {z}} \fr {t}$ can be thought of as the image of $d m (\fr {t})$ under the map $\gam \in G$ with the property $\gam \bbh^k=\fr {z}$; cf. \cite [p. 16]{Matt}. Specifically, the following statement holds.

\begin {lemma}\label {lzaq} Let $\fr {z}=\pi (\z)$, $\z \in \Gam_B (n,k)$. We write  $\z= \eta +v$,  $\eta \in G_{n,k}$, $v \in \eta^\perp \cap B_n$, and let $\gam \in K$ be a rotation satisfying
\be\label {iuae}
\gam:  \bbr^k \mapsto \eta, \qquad \gam : e_{n-k}\mapsto v/|v|.\ee
Denote
\be\label {iuae1}
f_{\fr {z}}(\fr {t})=f( \gam g_{n-k, n+1} (r) \fr {t}), \quad r= ||\fr {z}||= d(\fr {z}, \tilde o),\ee
where $g_{n-k, n+1} (r)$  is the hyperbolic rotation (\ref{hipr}) with $d$ replaced by $k$. Then
\bea\label {iu3e} (R_H f)(\fr {z})  &=& \intl_{\Gam_H (k,j)} f_{\fr {z}}(\fr {t}) \, d m (\fr {t})\qquad\\
\label {iu3e1}&=&\sigma_{k-j-1}\intl^\infty_0 d\nu_{k,j}(\rho) \intl_{SO(k)}f_{\fr {z}}(\a g_{n-j, n+1} (\rho) \bbh^j)\,
 d\a, \qquad \eea
\[
 d\nu_{k,j}(\rho)=(\sinh\, \rho)^{k-j-1}(\cosh \,\rho)^j \, d\rho, \qquad \rho = ||\fr {t}||= d(\fr {t}, \tilde o).\]
\end{lemma}
\begin {proof} Formula (\ref{iu3e}) gives precise meaning to $(R_H f)(\fr {z})$ under the assumption that the  right-hand side exists in the Lebesgue sense.
The more detailed formula (\ref{iu3e1}) follows from (\ref{iu3e})
 if we apply (\ref{jua}) with obvious changes. To complete the proof, it remains to show that (\ref{iu3e}) is independent of the choice of $\gam$ in (\ref{iuae1}).

We first note that the measure  $d m (\fr {t})$  in (\ref{iu3e}) is  $SO(k)$-invariant. Further, let
  $f'_{\fr {z}}(\fr {t})=f( \gam' g_{n-k, n+1} (r) \fr {t})$ and  $f''_{\fr {z}}(\fr {t})=f( \gam'' g_{n-k, n+1} (r) \fr {t})$, $ \gam'$ and $ \gam''$ being arbitrary rotations satisfying (\ref{iuae}). Then
 \[
 \intl_{\Gam_H (k,j)} f'_{\fr {z}}(\fr {t}) \, d m (\fr {t})= \intl_{\Gam_H (k,j)}f( \gam'' \b g_{n-k, n+1} (r) \fr {t}) \, d m (\fr {t}),\]
where $ \b=(\gam'')^{-1} \gam'$  preserves both $\bbr^k$ and $e_{n-k}$, and therefore commutes with $g_{n-k, n+1} (r)$.  Hence the right-hand side is
\[
\intl_{\Gam_H (k,j)}f( \gam''  g_{n-k, n+1} (r) \b\fr {t}) \, d m (\fr {t}).\]
We can write $\b$ in the form $\b=\b_1\b_2$ where $\b_1\in SO(n-k)$, $\b_2 \in SO(k)$, so that $\b\fr {t}=\b_2\fr {t}$. The latter allows us to write the above integral as
\bea
\intl_{\Gam_H (k,j)} \!\!\!f( \gam''  g_{n-k, n+1} (r) \b_2\fr {t}) \, d m (\fr {t})&=&\!\!\intl_{\Gam_H (k,j)}\!\!\!f( \gam''  g_{n-k, n+1} (r)\fr {t}) \, d m (\fr {t})\nonumber\\
&=&
\!\!\intl_{\Gam_H (k,j)} \!\!\!f''_{\fr {z}}(\fr {t}) \, d m (\fr {t}),\nonumber\eea
as desired. Here the first equality holds by the $SO(k)$-invariance of the measure  $d m (\fr {t})$ in $\Gam_H (k,j)$.
\end{proof}

Let us  connect representations of  $ R_H f$ in the hyperboloid model and the Beltrami-Klein one. We denote
\bea\label {lity}
(Mf)(\t)\!\!&=&\!\!(1\!-\!|\t|^2)^{-(k+1)/2} \,(f\circ \pi)(\t), \qquad \t \in \Gam_B (n,j), \quad\\
\label {lity1} (N\vp)(\fr {z})\!\!&=&\!\!  (\ch ||\fr{z} ||)^{-j-1} \,(\vp\circ \pi^{-1})(\fr {z}),\qquad \, \fr {z} \in \Gam_H (n,k),\quad\eea
where the map $\pi$ is defined by (\ref{mnqe}).
The corresponding inverse operators have the form
\bea \label {vlity} (M^{-1} F)(\fr{t})&=&(\ch \, ||\fr{t} ||)^{-k-1} (F\circ \pi^{-1})(\fr {t}), \qquad \fr{t} \in \Gam_H (n,j), \qquad\\
 \label {vlity1} (N^{-1} \Phi)(\z) &=& (1\!-\!|\z|^2)^{-(j+1)/2} \, (\Phi\circ \pi)(\z), \qquad \z \in \Gam_B (n,k).\qquad\eea

\begin{theorem}\label {afor} We have
\be\label {lity2}
R_H f =NR_B Mf,  \qquad  R_B F =N^{-1}R_H M^{-1} F,\ee
provided that  integrals in these equalities exist in the Lebesgue sense.
\end{theorem}
\begin{proof}
This theorem is an alternative version of Proposition 3.2 of Ishikawa \cite{I20}, which was given  in the group-theoretic  terms  for compactly supported smooth functions.
Our proof is different in principle, and the assumptions for functions are minimal.
We write (\ref{iu3e1}) as
\[
(R_H f)(\fr {z})=\sigma_{k-j-1}\intl^\infty_0 d\nu_{k,j}(\rho) \!\!\!\intl_{SO(k)}\!\!\! f_{\fr {z}} \left(\a [\hn \cap \{\bbr^j+ e_{n-j}\,\th \rho+ e_{n+1}\}]\right  )\,d\a\]
(cf. notation in (\ref{krtg1})), where $ \{\ldots \}$ stands for the smallest linear subspace containing the set in braces.
Let us  transform the interior integral by changing variable   $\a \mapsto \a\b$, $\b \in SO(k-j)$, where $SO(k-j)$ is the orthogonal
group of the subspace $\bbr^{k -j}= \bbr e_{n-k+1}\oplus\cdots\oplus  \bbr e_{n-j}$.
Integrating in $\b$, we obtain
\bea
\intl_{SO(k)}(...)&=& \intl_{SO(k)}\!\!\!d\a \!\!\!\intl_{SO(k-j)} \!\!\!f_{\fr {z}}(\a\b\, [\hn \cap \{\bbr^j+ e_{n-j}\,\th \rho + e_{n+1}\}])\,d\b\nonumber\\
&=& \intl_{SO(k)}\!\!\!d\a \!\!\!\intl_{SO(k-j)}\!\!\! f_{\fr {z}}(\a\, [\hn \cap \{\bbr^j + \b e_{n-j}\,\th \rho + e_{n+1}\}])\,d\b\nonumber\\
&=& \frac{1}{\sigma_{k-j-1}} \intl_{SO(k)}\!\!\!d\a \!\!\!\intl_{S^{k-j-1}}\!\!\! f_{\fr {z}}(\a\, [\hn \cap \{\bbr^j + \sig\, \th \rho + e_{n+1}\}])\,d\sig.\nonumber\eea
Hence
\bea
&&(R_H f)(\fr {z})=\intl^\infty_0 \!\!d\nu_{k,j}(\rho) \!\!\!\intl_{SO(k)}\!\!\! d\a\!\!\!\intl_{S^{k-j-1}}\!\!\!  f_{\fr {z}}(\a \,[\hn \cap \{\bbr^j + \sig \th \rho + e_{n+1}\}])\,d\sig\nonumber\\
&&=\intl^\infty_0 \left (\frac{\sh \, \rho}{\ch \, \rho}\right )^{k-j-1} (\ch \, \rho)^{k-1}\, d\rho
 \!\!\!\intl_{SO(k)}\!\!\! d\a\!\!\!\intl_{S^{k-j-1}}\!\!\!f_{\fr {z}}(...)\,d\sig\nonumber\\
&& \text {\rm (set $\th \rho =s$)}\nonumber\\
&&=\intl_0^1 \frac{s^{k-j-1}\, ds}{(1-s^2)^{(k+1)/2}}\!\!\!\intl_{SO(k)}\!\!\! d\a\!\!\!\intl_{S^{k-j-1}}\!\!\! f_{\fr {z}}(\a\, [\hn \cap \{\bbr^j + \sig s + e_{n+1}\}])\,d\sig\nonumber\\
&&=\intl_{SO(k)}\!\!\! d\a\!\!\!\intl_{B_{k-j}} \frac{f_{\fr {z}}(\a \,[\hn \cap \{\bbr^j + y + e_{n+1}\}])}{(1-|y|^2)^{(k+1)/2}}\, dy, \qquad B_{k-j}\!=\!B_n \cap \bbr^{k-j}.
\nonumber\eea
By (\ref{iuae1}),
\bea
f_{\fr {z}}(...)&=&f( \gam g_{n-k, n+1} (r) \a \,[\hn \cap \{\bbr^j + y + e_{n+1}\}])\nonumber\\
&=&f( \gam [\hn \cap \{\a g_{n-k, n+1} (r) (\bbr^j + y + e_{n+1})\}])\nonumber\\
&=&f( \gam [\hn \cap \{\a (\bbr^j + y + g_{n-k, n+1} (r) e_{n+1})\}])\nonumber\\
&=&f( \gam \a [\hn \cap \{\bbr^j + y + e_{n-k}\, \sh \,r + e_{n+1} \ch \,r\}]).\nonumber\eea
This gives
\bea
&&(R_H f)(\fr {z})=\intl_{SO(k)}\!\!\! d\a\nonumber\\
&&\times \!\!\intl_{B_{k-j}}\!\!\!\frac{ f( \gam \a [\hn \cap \{\bbr^j + y + e_{n-k}\, \sh \,r + e_{n+1} \ch \,r\}])\,dy}{(1-|y|^2)^{(k+1)/2}}\nonumber\\
&&  \text {\rm (set $y= z\,\ch \, r$) }\nonumber\\
&&=\frac{1}{(\ch \,r)^{j+1}}\intl_{SO(k)}\!\!\! d\a\!\!\!\intl_{B_{k-j} (1/\ch \,r)}\!\!\! \frac{(f\circ \pi)(\gam \a \,[\bbr^j +  e_{n-k}\, \th \,r +z])\,dz}{(1-\th^2 r -|z|^2)^{(k+1)/2}},\nonumber\eea
where
$B_{k-j} (1/\ch \,r)=\{z \in B_{k-j}:\, |z|<1/\ch \,r\}$.

We recall that $\fr {z} =\pi (\z)$, $\z=\eta +v\in \Gam_B (n,k)$, where $v\in \eta^\perp$ and $|v| =|\z|=\th \,r$.
Denote
\[B_{k} =B_n \cap \bbr^k, \quad B_{k} (1/\ch \,r)=\{z \in B_{k}:\, |z|<1/\ch \,r\}.\]
 In this notation, the above expression can be written in the Grassmannian language as
\[
(R_H f)(\fr {z})\!=\!(1-|v|^2)^{(j+1)/2}\intl_{G_{k,j}}\!\! d\sig\!\!\intl_{\sig^\perp \cap B_{k} (1/\ch \,r)}\!\!\!\!\!\frac{(f\circ \pi)(\gam [\sig +  e_{n-k} |v| +z])}{(1-|v|^2-|z|^2)^{(k+1)/2}}\, dz.\]
Set $\t=\sig +  e_{n-k} |v| +z$, $\;\xi=\gam \sig \subset \gam \bbr^k=\eta$, $\;\gam e_{n-k}=v/|v|$; cf. the definition of $\gam$ in (\ref{iuae}). Then $|\t|^2= |v|^2 +|z|^2$, and we have
\[
(R_H f)(\pi (\z))\!=\!(1\!-\!|\z|^2)^{(j+1)/2}\intl_{\xi\subset \eta} d_\eta \xi \intl_{\xi^\perp \cap \,\eta} \psi (\xi +v+z)\, dz,\]
\[ \psi (\t) =\frac{(f\circ \pi)(\t)}{(1-|\t|^2)^{(k+1)/2}}=(Mf)(\t);\]
cf. (\ref {lity}), (\ref{lasz}).  This gives the result.
\end{proof}

\subsubsection {The Dual Transform}  Our next objective is an analogue of Theorem \ref{afor} for the dual Radon transform $R_H^*$. The latter was defined by (\ref{Spec1}) in the group-theoretic terms.
Below we proceed in a different way. Specifically, we take  the dual chord transform $R_B^*$ (see  (\ref{lch5yasz})) as a point of departure
 and arrive at the same duality (\ref{wbcdhyp}), as in the group-theoretic terms, but under minimal assumptions for functions.

\begin{theorem}\label {aford} Let
\bea\label {lityd}
(P\vp)(\z)\!\!&=&\!\! (1\!-\!|\z|^2)^{(j-n)/2} \, \,(\vp\circ \pi)(\z), \qquad \z \in \Gam_B (n,k),\qquad\\
\label {lity1d} (Q\psi)(\fr{t})\!\!&=&\!\!(\ch \, ||\fr{t} ||)^{k-n} \,(\psi\circ \pi^{-1})(\fr{t}),\qquad \fr{t} \in \Gam_H (n,j).\qquad\eea
We set
\be\label {lity2q}
R_H^*\vp =QR^*_B P \vp.\ee
Then, as in (\ref{wbcdhyp}),
\be\label {wxade2}
 \intl_{\Gam_H (n,k)}\!\! \! (R_H f)(\fr {z})\,  \vp (\fr {z})\, d\fr {z}=\intl_{\Gam_H (n,j)} \!\! \! f(\fr {t})\, (R_H^*\vp)(\fr {t})\, d\fr {t},\ee
 provided that either side of  this equality is finite if $f$ and $\vp$ are replaced by $|f|$ and $|\vp|$, respectively.
\end{theorem}
\begin{proof}
Applying (\ref{ctew}), we have
 \[
 \intl_{\Gam_B (n,k)} (R_B f_1)(\z) \,\vp_1(\z)\, d \z=\intl_{\Gam_B (n,j)} f_1(\t)\, (R^*_B \vp_1)(\t)\, d\t,\]
 assuming that    the above integrals exist in the Lebesgue sense. By Lemma \ref {lem1}, this equality can be written as
\be\label {lide2} \intl_{\Gam_H (n,k)}\!\! \! \frac{ (R_B f_1)(\pi^{-1}(\fr {z})) \,\vp_1 (\pi^{-1}(\fr {z}))}{ (\ch \,||\fr {z}||)^{n+1}}\, d\fr {z}=\!\!
\intl_{\Gam_H (n,j)} \!\! \!\frac{ f_1(\pi^{-1}(\fr {t}))\, (R^*_B \vp_1)(\pi^{-1}(\fr {t}))}{ (\ch \,||\fr {t}||)^{n+1}}\, d\fr {t}.\ee
By (\ref{lity2}),
\bea
&&(R_B f_1)(\pi^{-1}(\fr {z}))=(N^{-1}R_H M^{-1} f_1)(\pi^{-1}(\fr {z}))\nonumber\\
&&=(1\!-\!|\pi^{-1}(\fr {z})|^2)^{-(j+1)/2} \,(R_H M^{-1} f_1)(\fr {z})= (\ch \,||\fr {z}||)^{j+1} (R_H f)(\fr {z}),\nonumber\eea
where
\be\label {lswde2}
f(\fr {t})=  (M^{-1} f_1)(\fr {t})=  (\ch \,||\fr {t}||)^{-k-1} f_1 (\pi^{-1}(\fr {t})).\ee
Hence the left-hand side of (\ref{lide2}) can be written as
\be\label {lswde3}
 \intl_{\Gam_H (n,k)}\!\! \! (R_H f)(\fr {z})\,  \vp (\fr {z})\, d\fr {z}, \qquad \vp (\fr {z})= \frac{ \vp_1 (\pi^{-1}(\fr {z}))}{ (\ch \,||\fr {z}||)^{n-j}}.\ee

For the right-hand side of (\ref{lide2}), owing to (\ref  {lswde2})  and (\ref  {lswde3}), we have
\bea
&&\intl_{\Gam_H (n,j)} \!\! \! \frac{f(\fr {t})\,(R^*_B \vp_1)(\pi^{-1}(\fr {t}))}
 {(\ch \,||\fr {t}||)^{n-k}}\, d\fr {t} \nonumber\\
&&\label {lswde4} =\intl_{\Gam_H (n,j)} \!\! \!f(\fr {t})\, \frac{(R^*_B [(\vp \circ \pi)(\z) (1-|\z|^2)^{(j-n)/2}])(\pi^{-1}(\fr {t}))}
 {(\ch \,||\fr {t}||)^{n-k}}\, d\fr {t}.\nonumber\eea
Thus, setting
\be\label {lswde4}
(R_H^*\vp)(\fr {t})= \frac{(R^*_B [(\vp \circ \pi)(\z) (1-|\z|^2)^{(j-n)/2}])(\pi^{-1}(\fr {t}))}
 {(\ch \,||\fr {t}||)^{n-k}},\ee
 we can write (\ref{lide2}) in the standard duality form  (\ref{wxade2}).
\end{proof}

 \subsection {Radon Transforms of Zonal Functions}

We recall that
a function $f$ on $\Gam_H (n, d)$  is called zonal if it is $K$-invariant. Every  such function
 is represented as a  single-variable function of the geodesic distance from the origin $\tilde o$ to the given $d$-geodesic.
 By $K$-invariance, the Radon transforms  $R_H f$ and $R^*_H \vp$  of zonal functions are zonal, too.
  Theorems \ref{afor} and \ref{aford}, combined with Lemma \ref {iant} allow us to obtain explicit expressions
  for these transforms in terms of  Erd\'{e}lyi-Kober type
fractional integration operators $I^{\a}_{\pm, 2}$; see Appendix.

\begin{lemma}\label {nsfo} {\rm (cf. \cite[Lemma 2.3]{Ru04d})}  Suppose that $\fr {t} \in \Gam_H (n,j)$, $\fr {z} \in \Gam_H (n,k)$, $0\le j<k\le n-1$, and let
$f(\fr {t})=f_1 (\ch \, ||\fr {t}||)$,  $\vp(\fr {z})=\vp_1 (\sh \, ||\fr {z}||)$.
 Then $(R_H f)(\fr {z})=F_1 (\ch \, ||\fr {z}||)$, $(R_H^*\vp)(\fr {t})=\Phi_1 (\sh \, ||\fr {t}||)$,
 where
\bea
 F_1 (s)&=&\frac {\sig_{k-j-1}}{s^{k-1}} \intl_{s}^\infty \! f_1(r)\,(r^2 -s^2)^{(k-j)/2 -1}  \,r^j \,dr\nonumber\\
 \label {zqwasz} &=& \pi^{(k-j)/2} s^{1-k} (I^{(k-j)/2}_{-,2} r^{j-1} f_1)(s),\eea
\bea\Phi_1 (r)&=&\frac{\sig_{k-j-1} \,
\sig_{n-k-1}}{\sig_{n-j-1}\, r^{n-j-2}} \intl_0^{r} \!\vp_1 (s)
(r^2\!-\!s^2)^{(k-j)/2 -1} s^{n-k-1} ds\nonumber\\
\label {qswasz}&=&\frac{\pi^{(k-j)/2} \,\sig_{n-k-1}}{\sig_{n-j-1}} \, r^{j+2-n}  (I^{(k-j)/2}_{+,2} s^{n-k-2} \vp_1)(r).\eea
It is assumed  that the corresponding
integrals on the right-hand sides exist in the Lebesgue sense.
\end{lemma}
\begin{proof}
 Because $f$ is zonal  on $\Gam_H (n,j)$,  $f\circ \pi$ is radial on $\Gam_B(n,j)$, and therefore $(f\circ \pi)(\t) = f_0 (|\t|) $ for some function $f_0$. Hence (see (\ref{lity2}))
\[(Mf)(\t)=(1\!-\!|\t|^2)^{-(k+1)/2} f_0 (|\t|), \] and  (\ref{nfty}) yields
\[
(R_B Mf)(\z)=\sig_{k-j-1}\intl_s^1 \!f_0(r)\, \frac {(r^2 -s^2)^{(k-j)/2 -1}  \,r dr}{(1-r^2)^{(k+1)/2}},\qquad s=|\z|.\]
This gives
\bea
&&(R_H f)(\fr {z})=(NR_B Mf)(\fr {z}) \nonumber\\
&&= \sig_{k-j-1}(1\!- \!\th^2 ||\fr{z} ||)^{(j+1)/2}\intl_s^1 \!f_0(r)\, \frac {(r^2 -s^2)^{(k-j)/2 -1}  \,r dr}{(1-r^2)^{(k+1)/2}}\nonumber\eea
with $s=\th ||\fr{z} ||$. Changing variables
\[
1-r^2 = 1/r_1^2,  \qquad 1-s^2 = 1/s_1^2 \]
and noting that $1\!- \!\th^2 ||\fr{z} ||=  1/\ch^2  ||\fr{z} ||$, we obtain the result in slightly different notation.

Let us prove the second statement.
Because $\vp$ is zonal  on $\Gam_H (n,k)$, $\vp \circ \pi$ is radial on $\Gam_B(n,k)$ and $(\vp \circ \pi)(\z) = \vp_0 (|\z|)$ for some function $\vp_0$. Hence by (\ref{lity2q}),
\[(P \vp)(\z)=(1\!-\!|\z|^2)^{(j-n)/2} \vp_0 (|\z|), \]
and therefore, by  (\ref{nfty1}),
\[
(R^*_B P \vp)(\t)=\frac{c}{r^{n-j-2}}\intl_0^r \!\vp_0 (s)
\frac {(r^2\!-\!s^2)^{(k-j)/2 -1} s^{n-k-1} ds}{(1-s^2)^{(n-j)/2}},\]
 where $r=|\t|$, $c= \sig_{k-j-1} \,\sig_{n-k-1}/\sig_{n-j-1} $ . This gives
\bea
(R_H^*\vp)(\fr {t})\!&=&\!\!(QR^*_B P \vp)(\fr {t}) \nonumber\\
\label {qse3sz} \!&=& \!\! \frac{c\,(1-r^2)^{(n-k)/2}}{r^{n-j-2}} \intl_0^r \!\vp_0 (s)
\frac {(r^2\!-\!s^2)^{(k-j)/2 -1} s^{n-k-1} ds}{(1-s^2)^{(n-j)/2}}\qquad \eea
with $r=\th ||\fr{t} ||$. Then we  change variables
\[
1-r^2 = 1/(1+r_1^2),  \qquad 1-s^2 = 1/(1+s_1^2), \]
and note that $\th^2 r = \sh^2 r / (1+\sh^2 r)$. It remains to adjust the notation.
\end{proof}

\begin{remark}\label {nusfo5}
 The integral (\ref{zqwasz}) is finite for almost all $s>0$ provided that
\be\label {vdffp0} \intl_a^\infty |f_1 (r)|\, r^{k-2} dr <\infty\ee
 for all $a>0$.
  The integral (\ref{qswasz}) is finite for almost all $r>0$ if
 $s \mapsto s^{n-k-1} \vp_1(s)$ is a locally integrable function on $\bbr_{+}$. Both assumptions are sharp; see, e.g.,
 \cite [Lemma 2.42]{Ru15}.
\end{remark}

\begin{example}\label {nusfo51}
Let $a>1$,
\[f(\fr {t})= \frac{(a^2 - \ch^2 \, ||\fr {t}||)_+^{\a/2 -1}}{\ch^{\a +k-1} \, ||\fr {t}||},  \qquad   Re \, \a >0.\]
  Then, by (\ref{zqwasz}) with $s= \ch \, ||\fr {z}||$,
\bea
(R_H f)(\fr {z})&=&\frac {\sig_{k-j-1}}{s^{k-1}} \intl_{s}^a  \frac{(a^2 - r^2)^{\a/2 -1} (r^2 -s^2)^{(k-j)/2 -1}}{r^{\a+k-j-1}} \,dr\nonumber\\
&=&\frac {c_\a (a^2 - \ch^2 \, ||\fr {z}||)^{(\a+k-j)/2 -1}}   {\ch^{\a +k-1} \, ||\fr {z}||}, \quad c_\a=\frac{\pi^{(k-j)/2} \Gam (\a/2)}{a^{k-j} \Gam ((\a+k-j)/2)}; \nonumber\eea
use, e.g., \cite [2.2.6 (2)] {PBM1}.
Now we set $\a=2$, $a= \ch \,b$, and use the duality (\ref{wxade2}). This gives

\be\label{yuyu}
\intl_{ ||\fr {t}|| < b } \!\! \! \frac{(R_H^*\vp)(\fr {t})}{\ch^{k+1} \, ||\fr {t}||} \, d \fr {t} =
c_2 \intl_{ ||\fr {z}|| < b } \!\! \!  \vp (\fr {z})\, \frac{(\ch^2 b - \ch^2 \, ||\fr {z}||)^{(k-j)/2}}{\ch^{k+1} \, ||\fr {z}||}\,  d\fr {z}.\ee
\[ c_2=\frac{\pi^{(k-j)/2}}{\Gam ((k-j)/2 +1)\, \ch^{k-j} b}.\]
\end{example}

\subsection {Existence in the Lebesgue Sense. The General Case}

By Theorem \ref  {afor},  $R_H f$  exists in the Lebesgue sense if and only if so does $R_B Mf$.
The following statement contains precise information about  the behavior  $R_H f$ and its existence.

\begin{lemma} \label {rolla} Let $\fr {z}\in \Gam_H (n,k)$, $\fr {t}\in \Gam_H (n,j)$, $0\le j<k\le n-1$.  Denote $|| \fr {z} ||=d (\fr {z}, \tilde o)$, $|| \fr {t} ||=d (\fr {t}, \tilde o)$. Then
\be\label {ulas}
\intl_{\Gam_H (n,k)} \!\!\! (R_H f)(\fr {z})\, d\fr {z}= \!\!\intl_{\Gam_H (n,j)} \!\!\! f(\fr {t})\, d\fr {t},\ee
\be\label {ulas1}
\intl_{\Gam_H (n,k)} \!\frac{(R_H f)(\fr {z})\,d\fr {z}}{  (\ch\, || \fr {z} ||)^{n-j}}=\intl_{\Gam_H (n,j)} \! \frac{f(\fr {t})\,d\fr {t}}{  (\ch\, || \fr {t} ||)^{n-k}},\ee
provided that  integrals on the right-hand side exist in the Lebesgue sense. More generally, if
\[
 u(\fr {z})=\frac{(\th\, || \fr {z} ||)^{\a+k-n}}{(\ch\, || \fr {z} ||)^{n-j}}, \qquad  v(\fr {t})=\frac{(\th\, || \fr {t} ||)^{\a+k-n}}{(\ch\, || \fr {t} ||)^{n-k}},  \qquad Re \,\a> 0,\]
then
\be\label {tulasa}
\intl_{\Gam_H (n,k)} (R_H f)(\fr {z})\, u(\fr {z})\, d\fr {z}=\lam_2 \intl_{\Gam_H (n,j)} f(\fr {t})\, v(\fr {t})\, d\fr {t},\ee
\[
\lam_2 =\frac{ \Gam (\a/2) \, \Gam ((n-j)/2)}{\Gam ((\a+k -j)/2) \, \Gam ((n-k)/2)}. \]
\end{lemma}
\begin{proof}
By (\ref{ally}), for any $\psi\in L^1 (\Gam_B (n,j))$,
\[\intl_{\Gam_B (n,k)}\,  (R_B \psi)(\z)\,  d \z=\intl_{\Gam_B (n,j)}  \psi(\t)\, d\t.\]
Hence, by (\ref{kred}),
\[\intl_{\Gam_H (n,k)} \frac{ (R_B \psi)(\pi^{-1}(\fr {z})) \,d\fr {z} }{ (\ch \,||\fr {z}||)^{n+1}}=\intl_{\Gam_H (n,j)} \frac{\psi (\pi^{-1}(\fr {t})) \,d\fr {t} }{ (\ch\, ||\fr {t}||)^{n+1}}.\]
By (\ref{lity2}), setting  $\psi =Mf$, we have
\[
\frac{(R_B \psi)(\pi^{-1}(\fr {z}))}{(\ch\, ||\fr {z}||)^{n+1}}=\frac{(N^{-1} R_H M^{-1} \psi)(\pi^{-1}(\fr {z}))}{(\ch\, ||\fr {z}||)^{n+1}}=\frac{(R_H f)(\fr {z})}{(\ch\, ||\fr {z}||)^{n-j}}.\]
 Further, by (\ref {lity}),
\[\frac{\psi (\pi^{-1}(\fr {t}))}{ (\ch \,||\fr {t}||)^{n+1}}=(1-\th^2\,  ||\fr {t}||)^{(n-k)/2} f(\fr {t})=\frac{ f(\fr {t})}{(\ch \,||\fr {t}||)^{n-k}}.\]
This gives   (\ref{ulas1}). The equalities (\ref{ulas}) and  (\ref{tulasa})  follow from  (\ref {uwaff}) and (\ref {waff}) in a  similar way.
\end{proof}

 Note that  powers of $\ch\, || \fr {z} ||$  and  $\ch\, || \fr {t} ||$ in the above equalities give precise
 information about  the behavior of the corresponding functions when $|| \fr {z} ||$ and $|| \fr {t} ||$ tend to  infinity.
 Similarly,  powers of $\th\, || \fr {z} ||$  and  $\th\, || \fr {t} ||$  reflect  the behavior of functions for small $|| \fr {z} ||$ and $|| \fr {t} ||$.

We denote
 \be\label {ented}
 L^1_\lam (\Gam_H (n,d))=\Big\{ f: ||f||_{1,\lam}\equiv \intl_{\Gam_H (n,d)}\!\! \!  \frac{|f (\fr {t})|}{ (\ch\, ||\fr {t}||)^{\lam}}\, d\fr {t}<\infty \Big\}.\ee
In this notation, (\ref{ulas1}) yields the following statement.
\begin{theorem} Let  $0\le j<k\le n-1$. Then $R_H$ is a linear bounded operator from  $L^1_{n-k} (\Gam_H (n,j))$  to $L^1_{n-j} (\Gam_H (n,k))$.
 In particular, for $f\in L^1_{\lam} (\Gam_H (n,j))$, the Radon transform
$(R_H f)(\fr {z})$ exists for almost all $\fr {z} \in  \Gam_H (n,k)$  in the Lebesgue sense provided that $\lam \ge n-k$, where the bound $n-k$ is sharp.
\end{theorem}

The following statement characterizes the existence of $R_H f$ on $L^p$ functions.
\begin {proposition}\label {waff2x}
If $f \in L^p (\Gam_H (n,j)), \; 1 \le p < (n-1)/(k -1)$,
then $(R_H f)(\fr {z})$ is finite for almost all $\fr {z} \in \Gam_H  (n,k)$.
\end{proposition}
\begin{proof} By H\"older's inequality, the right-hand side of (\ref  {ulas1}) does not exceed $c \, ||f||_p$, where
\[ c^{p'}=\intl_{\Gam_H (n,j)} \! \frac{d\fr {t}}{  (\ch\, || \fr {t} ||)^{(n-k)p'}}, \qquad 1/p +1/p' =1.\]
By (\ref{aick1a}),
\bea c^{p'}&=&\sigma_{n-j-1}\intl_1^\infty (s^2 -1)^{(n-j)/2 -1} s^{j -(n-k)p'}\,  ds\nonumber\\
&=& \frac{\sigma_{n-j-1}}{2} B \left (\frac{n-1 -p\, (k-1)}{2\,(p-1)}, \frac{n-j}{2}\right )<\infty\nonumber\eea
if $1 \le p < (n-1)/(k -1)$.
\end{proof}

\begin{remark} The bound for $p$  in Proposition \ref {waff2x} is sharp.  For instance, if $p \ge (n-1)/(k -1)$, then
 \be\label {atem} f(\fr {t})= (\ch\, || \fr {t} ||)^{(1-n)/p}\, (\log (1 +\ch\, || \fr {t} ||))^{-1} \ee
 belongs to $L^p (\Gam_H (n,j))$ but $(R_H f)(\fr {z}) \equiv \infty$.
\end{remark}

The next statement gives a simple sufficient condition of the pointwise convergence of the integral   $(R_H f)(\fr {z})$.

\begin {proposition}\label {wareqs}  Let $0\le j<k\le n-1$. If
\be\label {rgess} A_h \equiv {\rm ess} \!\!\!\!\sup\limits_{ \fr {t} \in \Gam_H (n,j)} (\ch  ||\fr {t}||)^\lam |f(\fr {t})|<\infty \quad \text {for some $\lam > k-1$},\ee
 then  $(R_H f)(\fr {z})$ is finite for  all $\fr {z} \in \Gam_H  (n,k)$. The condition  $\lam > k-1 $  is sharp.
\end{proposition}
\begin{proof} Using (\ref{rgess}) and Lemma \ref{nsfo},  for $f_1(\fr {t})= (\ch  ||\fr {t}||)^{-\lam}$ we have
\bea
|(R_H f)(\fr {z})| &\le& A_h (R_H f_1)(\z)|= \frac {A_h \,\sig_{k-j-1}}{s^{k-1}} \intl_{s}^\infty \! r^{j-\lam}\,(r^2 -s^2)^{(k-j)/2 -1}\,dr\nonumber\\
&=& c\, s^{-\lam}<\infty, \qquad   \lam > k-1,  \quad   s=\ch ||\fr {z}||.\nonumber\eea
This gives the result.
\end{proof}

Let us discuss the existence of the  dual transform  $(R^*_H \vp)(\fr {t})$.

\begin{theorem}\label {lrrbxW} Let  $ 0 \le j<k\le n-1$. The dual transform  $(R^*_H \vp)(\fr {t})$ is finite a.e.
 for every locally integrable function $\vp $ on $\Gam_H  (n,k)$ and represents a locally integrable function on $\Gam_H  (n,j)$. Further,
  if
 \be\label {rges7}  B_h \equiv {\rm ess} \!\!\!\!\sup\limits_{ \fr {z} \in \Gam_H  (n,k)} (\sh ||\fr {z}||)^\del |\vp(\fr {z})|<\infty, \qquad \del < n- k,\ee
   then  $(R^*_H \vp)(\fr {t} )$ is finite for  all $||\fr {t}||>0$. The condition  $\del < n- k$  is sharp.
 \end{theorem}
\begin{proof}
The first statement follows from (\ref{yuyu}). Indeed,  for any $b>0$,
\bea\label{zyuyu}
&&\intl_{ ||\fr {t}|| < b } \!\! \! |(R_H^*\vp)(\fr {t})| \, d \fr {t} \le (\ch \, b)^{k+1} \intl_{ ||\fr {t}|| < b } \!\! \! \frac{(R_H^*|\vp|)(\fr {t})}{(\ch \, ||\fr {t}||)^{k+1}} \, d \fr {t}\nonumber\\
&& =c_2 \,(\ch \, b)^{k+1}  \intl_{ ||\fr {z}|| < b } \!\! \!  |\vp (\fr {z})|\, \frac{(\ch^2 b - \ch^2 \, ||\fr {z}||)^{(k-j)/2}}{(\ch \, ||\fr {z}||)^{k+1}}\,  d\fr {z} \nonumber\\
&&\le \frac{\pi^{(k-j)/2} \, (\ch \, b)^{k+1}}{\Gam ((k-j)/2 +1)}   \intl_{ ||\fr {z}|| < b } \!\! \!  |\vp (\fr {z})|\, d\fr {z}. \nonumber\eea
If $\vp$ satisfies  (\ref{rges7}), then by (\ref{qswasz}), for $r= \sh ||\fr {t}||$ and $\vp_1 (\fr {z})= (\sh ||\fr {z}||)^{-\del}$  we have
\[ |(R^*_H \vp)(\fr {t})| \le B_h (R^*_H \vp_1)(\fr {t})= \frac{c\, B_h}{r^{n-j-2}}\intl_0^r  (r^2-s^2)^{(k-j)/2 -1} s^{n-k-\del-1} \,ds<\infty\]
 if $\del <n-k$, $r>0$.
\end{proof}

The following statement gives additional information about  the existence of $(R^*_H \vp)(\fr {t})$.

\begin{proposition} Let $\fr {z}\in \Gam_H (n,k)$, $\fr {t}\in \Gam_H (n,j)$, $0\le j<k\le n-1$, $Re \, \a >0$.  Then
\be\label {exis1}  \intl_{\Gam_H (n,j)} \!\! \! \frac{(R^*_H \vp)(\fr {t})}
{(\ch \,||\fr {t}||)^{k-1+\a}}\,  d\fr {t} = \lam_1 \! \!\intl_{\Gam_H (n,k)}\!\! \!  \frac{\vp (\fr {z})}{ (\ch\, ||\fr {z}||)^{k-1+\a}}\, d\fr {z},\ee
\be\label {exis} \intl_{\Gam_H (n,j)} \! \!\!\! \! (R^*_H \vp)(\fr {t})\,\frac{(\th \,||\fr {t}||)^{j-k-\a}}{(\ch \,||\fr {t}||)^{k+1}}\, d\fr {t} = \lam_1 \!\! \! \!\intl_{\Gam_H (n,k)}\!\! \!\! \!
 \vp (\fr {z}) \,\frac{(\th \,||\fr {z}||)^{-\a}}{(\ch \,||\fr {z}||)^{j+1}} \,  d\fr {z},\ee
\[ \lam_1 =\frac{\pi^{(k-j)/2} \, \Gam (\a/2)}{\Gam ((\a+k -j)/2)}. \]
These equalities hold provided that either side of them exists in the Lebesgue sense.
\end{proposition}
\begin{proof}
  By (\ref{vwaff1}),
 \[ \intl_{\Gam_B (n,j)} \!\!(R^*_B \psi)(\t) (1\! - \! |\t|^2)^{\a/2-1}\, d\t = \lam_1 \!\intl_{\Gam_B (n,k)}\!\!\! \psi (\z) (1 \! - \! |\z|^2)^{(\a+k
-j)/2 -1} d \z.\]
 Hence, by (\ref{kred}),
  \[\intl_{\Gam_H (n,j)} \!\! \! \frac{ (R^*_B \psi)(\pi^{-1}(\fr {t}))}
{(\ch \,||\fr {t}||)^{n-1+\a}}\,  d\fr {t} = \lam_1 \! \!\intl_{\Gam_H (n,k)}\!\! \!  \frac{\psi (\pi^{-1}(\fr {z}))}{ (\ch\, ||\fr {z}||)^{n-1+\a +k -j}}\, d\fr {z}.\]
By (\ref {lity2q}), the expression under the sign of the integral on the left-hand side is
\[\frac{(Q^{-1} R^*_H P^{-1} \psi)(\pi^{-1}(\fr {t}))}
{(\ch \,||\fr {t}||)^{n-1+\a}}=\frac{(R^*_H \vp)(\fr {t})}{ (\ch \,||\fr {t}||)^{k-1+\a}}\]
 where $\vp=P^{-1} \psi$. Setting $\psi =P\vp$ in the right-hand side, we obtain (\ref{exis1}).
The proof of (\ref{exis}) is similar and relies on (\ref{vwaff1D}).
\end{proof}

\subsection {Injectivity}

\begin{theorem}\label {iste5} If  $ 0 \le j<k\le n-1$,  $ j+k \le n-1$, then
the Radon transform $R_H$ is injective on  $ L^1_{\lam} (\Gam_H (n,j))$ for all  $\lam \ge n-k$ .
\end{theorem}
\begin{proof} Let $f \!\in \! L^1_{\lam} (\Gam_H (n,j))$, and let $(R_H f)(\fr {z})\!= \!0$ for almost all $\fr {z} \!\in \!\Gam_H (n,k)$.
Because  $ L^1_{\lam} (\Gam_H (n,j))$ is embedded in  $ L^1_{n-k} (\Gam_H (n,j))$ when  $\lam > n-k$, it suffices to assume $f \!\in \! L^1_{n-k} (\Gam_H (n,j))$. By Theorem \ref{afor}, we have
$(R_B Mf)(\z)=0$  for almost all $\z\!\in \!\Gam_B (n,k)$. Hence, if $Mf \in L^1(\Gam_B (n,j))$, then, by Theorem \ref {tally} (i),
it follows that
\[ (Mf)(\t)=(1-|\t|^2)^{-(k+1)/2} \,(f\circ \pi)(\t)=0\; \text {\rm    for almost all}\;  \t\in  \Gam_H (n,j), \]and therefore $f=0$ a.e. on $\Gam_H (n,j)$.
To complete the proof, it remains to note that $Mf \in L^1(\Gam_B (n,j))$ because, by (\ref{kred}) and (\ref{tat}),
\[
\intl_{\Gam_B (n,j)}  |(Mf)(\t)|\, d\t=\intl_{\Gam_H (n,j)}  \frac{|f (\fr {t})|}{ (\ch\, ||\fr {t}||)^{n-k}}\, d\fr {t}<\infty.\]
\end{proof}

An analogue of Theorem \ref{iste5}  for  $R^*_H$ reads as follows.
\begin{theorem}\label {iste51} If  $ 0 \le j<k\le n-1$,  $ j+k \ge n-1$, then
the dual Radon transform $R^*_H$ is injective on  $ L^1_{\del} (\Gam_H (n,k))$ for all  $\del \ge j+1$.
\end{theorem}
\begin{proof}   It suffices to consider the case $\del=j+1$.  We recall that by Theorem \ref {aford}, $R_H^*\vp =QR^*_B P \vp$, where
\[(P\vp)(\z)= (1-|\z|^2)^{(j-n)/2} \, \,(\vp\circ \pi)(\z),\quad (Q\psi)(\fr{t})= (\ch \, ||\fr{t} ||)^{k-n} \,(\psi\circ \pi^{-1})(\fr{t}).\]
If $\vp \!\in \! L^1_{j+1} (\Gam_H (n,k))$, then $P\vp \in L^1(\Gam_B (n,k))$ because by (\ref{kred}) and (\ref{tat}),
\[
\intl_{\Gam_B (n,k)}  |(P\vp)(\z)|\, d\z=\intl_{\Gam_H (n,k)}  \frac{|\vp (\fr {z})|}{ (\ch\, ||\fr {z}||)^{j+1}}\, d\fr {z}<\infty.\]
Suppose that $(R^*_H \vp)(\fr {t})\!= \!0$ for almost all $\fr {t} \!\in \!\Gam_H (n,j)$. Then $(R^*_B P \vp)(\t)\!= \!0$ for almost all $\t \!\in \!\Gam_B (n,j)$,
and therefore, by Theorem \ref {tally} (ii), $(P\vp)(\z)= 0$   for almost all $\z\!\in \!\Gam_B (n,k)$. The latter means that  $\vp (\fr {z})= 0$  a.e. on $\Gam_H (n,k)$.
This gives the result.

\end{proof}

\subsection {Support Theorems}\label {ider1}

 The following statement is a hyperbolic analogue of Theorem  \ref{thrt4}.

\begin{theorem} \label {Crf4h} Let  $\fr {z}\in \Gam_H (n,k)$, $\fr {t}\in \Gam_H (n,j)$, $0\le j<k\le n-1$, $j+k\le n-2$, $r>0$.
 The implication
\be\label{Ikvv1}
(R_H f)(\fr {z})=0\; \forall \, ||\fr {z}|| > r\; \Longrightarrow \; f(\fr {t})=0\; \forall \, ||\fr {t}|| > r\ee
 holds for any
  continuous function $f$ satisfying (\ref{rgess}).
 \end{theorem}

\begin{proof} The proof follows the same idea as in  Theorem  \ref{thrt4}.  We
fix any $(k+j+1)$-geodesic  $Z$ in $\hn$ at a distance $s>r$ from the origin $\tilde o$. Consider another $(k+j+1)$-geodesic, having the form
\[
Z_t= g_{n-k-j-1, n+1} (s)\,\bbh^{k+j+1} \]
where $g_{n-k-j-1, n+1} (s)$ is defined by (\ref{hipr}) (with $d=k+j+1$) and
\[\bbh^{k+j+1}\!=\hn \cap \, \bbe^{k+j+1,1},  \quad \bbe^{k+j+1,1} \!=\L (e_{n-j-k}, \ldots, e_{n+1})\!=\! (\bbr^{n-j-k-1})^\perp.\]
Let
$\rho_Z \in G$ be a hyperbolic rotation that sends $Z_t$ to $Z$.
Suppose that a $k$-geodesic $\fr {z}$ lies in $Z$ and denote
 $\fr {z}_1=\rho_Z^{-1} \fr {z} \subset Z_s$, $\; f_Z (\fr {t})=f(\rho_Z \fr {s})$.
By the assumption, $(R_{H} f)(\fr {z})=0$ for all $\fr {z} \subset Z$, and therefore
$(R_{H} f)(\rho_Z \fr {z}_1) =(R_{H} f_Z)(\fr {z}_1)=0$ for all $k$-geodesics $\fr {z}_1$ in $Z_s$.
Every $\fr {z}_1$ has the form $\fr {z}_1=g_{n-k-j-1, n+1} (s)\fr {z}_2 $, where $\fr {z}_2$ is a $k$-geodesic  in $\bbh^{k+j+1}$.
Hence, for any
$\fr {z}_2$  in $ \bbh^{k+j+1}$,
\[
0=(R_{H} f_Z)(g_{n-k-j-1, n+1} (s)\fr {z}_2)=\intl_{\fr {z}_2} f_Z (g_{n-k-j-1, n+1} (s)\fr {t})\, d\fr {t}= (R_{H} \tilde f)(\fr {z}_2).\]
The latter  is the geodesic $j$-to-$k$
 transform of the function
 \[\tilde f (\fr {t})= f_Z (g_{n-k-j-1, n+1} (s)\fr {t})\]
  on the hyperbolic space $\bbh^{n_1}$,  $n_1=j+k+1$.
 Because
   $j+k=n_1 -1$, by Theorem  \ref {iste5} it follows that $R_H$ is injective on  $L^1_{n_1-k} (\Gam_H(n_1,j))$, i.e., whenever
\[
 I=\intl_{\Gam_H (n_1,j)}  \frac{| \tilde f (\fr {t})|}{ (\ch\, ||\fr {t}||)^{n_1-k}}\, d\fr {t}=\intl_{\Gam_H (j+k+1,j)}  \frac{| \tilde f (\fr {t})|}{ (\ch\, ||\fr {t}||)^{j+1}}\, d\fr {t}<\infty. \]
   By (\ref{rgess}) and (\ref{aick1a}),
 \[
 I\le  c\intl_{\Gam_H (j+k+1,j)}  \frac{d\fr {t}}{ (\ch\, ||\fr {t}||)^{j+1+\lam}} =c_1 \intl_1^\infty  (s^2 -1)^{(k-1)/2}\, s^{-\lam -1}  < \infty\]
 if $\lam>k-1$, as in (\ref{rgess}).
Hence, by injectivity,  
\[\tilde f  (\fr {t})= f_Z (g_{n-k-j-1, n+1} (s)\fr {t}) =0\quad \text {\rm for all}\quad  \fr {t} \subset \bbh^{j+k+1}.\]
 Because $Z$ and $s>r$ are arbitrary, the result follows.
 \end{proof}

To the best of our knowledge, an analogue of Theorem \ref {Crf4h} for the case  $j+k= n-1$, $j>0$,  is unknown. The following conjecture looks plausible.
\begin{definition}  \label {abba1}
A $C^\infty $ function $f$ on $\Gam_H (n,d)$ is called rapidly decreasing if
 the quantity $(\ch ||\fr {t}||)^{m}  |f(\fr {t})|$ is bounded for all $m>0$.
 We denote by
  $S (\Gam_H (n,d))$ the space of all functions $f$ on $\Gam_H (n,d)$ which are rapidly decreasing together with all their derivatives.
\end{definition}

\begin{conjecture} \label {eed4h}  {\rm Let $0\le j<k \le n-1$, $j+k= n-1$,  $r>0$.
 If $f\in S (\Gam_H (n,d))$ and  $(R_H f)(\fr {z})=0$ for all $||\fr {z}|| > r$, then $f(\fr {t})=0$ for all $||\fr {t}|| > r$.}
\end{conjecture}

The following support theorem for $R^*_H$  is a consequence of the corresponding Theorem \ref {Ddt4}  for $R^*_B$.
We recall that by (\ref{lity2q}), $R_H^*\vp =QR^*_B P \vp$, where
\[ (P\vp)(\z) \!= \!(1-|\z|^2)^{(j-n)/2} (\vp\circ \pi)(\z), \quad (Q\psi)(\fr{t})\!=\!(\ch \, ||\fr{t} ||)^{k-n} (\psi\circ \pi^{-1})(\fr{t}),\]
\[\z \in \Gam_B (n,k), \qquad \fr{t} \in \Gam_H (n,j).\]

\begin{theorem} \label {Ddt4h}  Let  $0\le j<k \le n-1$, $j+k\ge n$,
 $\fr {t} \in  \Gam_H (n,j)$,  $\fr {z} \in \Gam_H(n,k)$, $r>0$.
 Then the implication
 \be\label{Ddkvv1h}
(R^*_H \vp)(\fr {t})=0\; \forall \, ||\fr {t}|| \in (0,r) \; \Longrightarrow \; \vp(\fr {z})=0\; \forall \,  ||\fr {z}|| \in (0,r)\ee
  holds for all   functions  $\vp$  which are  continuous on the set $\{ \fr {z}: 0<||\fr {z}||<r\}$ and satisfy
  \be\label {otaBh}
B_* \equiv \sup\limits_{0<||\fr {z}||<r} (\th ||\fr {z}||)^\del  (\ch ||\fr {z}||)^{n-j}    |\vp(\fr {z})|<\infty \quad \text {for some} \quad\del < n- k.\ee
\end{theorem}
\begin{proof} Suppose that $(R^*_H \vp)(\fr {t})=0$ whenever $||\fr {t}|| \in (0,r)$. It follows that $(R^*_B P \vp)(\t)=0$ for all $|\t|\in (0, \th \,r)$. Hence, by Theorem \ref {Ddt4},
$(P \vp)(\z)=0$ for all  $|\z|\in (0,\th \,r)$ provided that
\be\label{mdfrd}
\sup\limits_{0<|\z|<r} |\z|^\del |(P \vp)(\z)|<\infty \quad \text {for some} \quad\del < n- k.\ee
However, (\ref{mdfrd}) is equivalent to (\ref{otaBh}), because $|\z|=\th ||\fr {z}||$ and $1\!-\!|\z|^2= 1/\ch^2 ||\fr {z}||$. This completes the proof.
\end{proof}

Note that the factor $(\th ||\fr {z}||)^\del$ in (\ref{otaBh}) reflects admissible behavior of $\vp(\fr {z})$ as $||\fr {z}|| \to 0$ and $(\ch ||\fr {z}||)^{n-j}$ plays a similar  role for $||\fr {z}|| \to \infty$.

\subsection {Connection with  Rank-One Transforms}

Below, for the sake of convenience,  we  write $R_{j,k}$ in place of  $R_H$. The case $j=0$
corresponds to the  totally geodesic transform  $R_k : h \mapsto \int_\z h$ that  takes functions  on $\hn$ to functions on $\Gam_H (n,k)$;
cf. \cite {BC92, BCK, H11, I97, Ku94, Ru02c, Str81}. The corresponding dual transform is defined by
\be\label {concesw}
(R^*_k \psi)(x)=\intl_K \psi (r_x \gam \bbh^k)\, d\gam,\ee
where $r_x\in G=SO_0 (n,1)$, $\;r_x \tilde o =x$.

\begin{lemma} \label {esgu} Let $0\le j<k\le n-1$. Then
\be\label {conce2}
R_{j,k} R_j h =R_k h,\ee
 provided that the integral on the right-hand side exists in the Lebesgue sense.
\end{lemma}
\begin{proof} By (\ref{iuae})-(\ref{iu3e}),
\bea
 (R_{j,k} R_j h )(\fr {z})&=& \intl_{\Gam_H (k,j)} (R_j h)(\gam g_{n-k, n+1} (r) \fr {t}) \, d\fr {t} \nonumber\\
&=& \intl_{\Gam_H (k,j)} (R_j [h\circ \gam g_{n-k, n+1} (r)])( \fr {t}) \,d\fr {t}\nonumber\\
&=&{\text {\rm (use, e.g., (\ref{ulas}) with $j=0$ and $n$ replaced by $k$)}} \nonumber\\
&=&\intl_{\bbh^k } h( \gam g_{n-k, n+1} (r) \,x)\, dx=(R_k h)(x),\nonumber\eea
as desired.
\end{proof}

\subsubsection {Inversion Formulas}
To formulate the  inversion results,   we introduce  the following meromorphic family of dual hyperbolic sine transforms generalizing (\ref{concesw}):
\be\label {conce5}
(\overset {*} {R}_k {}^{\!\!\a} \varphi)(x) = \gamma_{n,k} (\a)  \intl_{\Gam_H (n,k)}\!\!\! \varphi (\xi)\, \big(\sinh \, d(x, \xi)\big)^{\a+k-n} d\xi,
\ee
\[
\gamma_{n, k} (\a)=\frac{2^{-\a-k} \, \Gamma ((n-\a-k)/2) \,
\Gamma ((n-k)/2)}{\pi^{n/2}\Gamma (n/2) \, \Gamma (\a/2)}, \qquad  Re \a >0, \]
\[\a + k -n\neq 0, 2, 4, \dots \, .\]
One can show  \cite[Lemma 4.1]{Ru02c} that if $\vp$ is a compactly supported continuous function, then
\[
\lim\limits_{\a\to 0} \overset {*} {R}_k {}^{\!\!\a} \varphi= c\,R_k^*\vp, \qquad  c\,=\,\frac{2^{-k}\Gamma \big( (n-k)/2\big)}{\pi^{k/2} \Gamma(n/2)}.\]
Thus the dual transform    (\ref{concesw}) can be formally included in (\ref{conce5}). Let $\Del$ be the Beltrami-Laplace operator
on $\hn$; see, e.g. \cite{Bray94}, \cite [p. 377]{Ru15}. We denote
\[
P_m(\Delta) = \prod\limits^m_{i=1} \big[-\Delta + (2i-n)(2i-1)\big]. \]
  Then Theorem A from  \cite {Ru02c} implies the following statement.

\begin{theorem} \label{qqwwa}  If $f \!=\! R_j h$, for some $h\in C_c^\infty (\hn)$,
then $f$ can be reconstructed from $\vp=R_{j,k}f$ by the formula
\be\label{ktyhu88} f=  R_j D_m \vp,  \ee
where the operator $D_m$ is defined as follows:

{\rm (i)}  \ If $n$ is odd, then

\be\label{kt8} D_m \vp =P_m(\Delta) \overset {*} {R}_k {}^{\!\! 2m-k} \vp\quad \forall m \ge k/2.\ee

{\rm (ii)} \ If $n$ is even, then (\ref{kt8}) is applicable only if $k/2\le m \le n/2-1$, and another  formula  holds:
\bea
 (D_m \vp)(x)\! &=&\!\! -P_{n/2} (\Delta) \Bigg[  \frac{ 2^{1-n}}{\pi^{n/2} \Gamma(n/2)}  \intl_{\Gam_H (n,k)}  \!\!\!\vp( \fr {z})\log \big(\sinh \, d(x,\fr {z})\big)\,d\fr {z}\Bigg]  \nonumber\\
&+&\frac{ (-1)^{n/2} \Gamma\big((n+1)/2\big)}{\pi^{(n+1)/2}} \intl_{\Gam_H (n,k)} \!\!\! \vp(\fr {z}) \ d\fr {z}.\nonumber\eea
\end{theorem}

\begin{proof} By (\ref{conce2}),    $\vp =R_{j,k}f = R_{j,k} R_j h =R_k h$. Hence, by  Theorem A from  \cite {Ru02c}, $h= D_m \vp$, and therefore,
$f=  R_j h=R_j D_m \vp$.
\end{proof}

A similar inversion result for Radon transforms on affine Grassmannians was obtained in \cite {RW21}.

\subsubsection {Support Theorem}
Lemma \ref{esgu} allows us to specify the support theorem in Subsection \ref{ider1} and Conjecture \ref{eed4h}
for $R_H f\equiv R_{j,k}f$ under the assumption that $f$ belongs to the range of $ R_j$.

\begin{lemma} \label {actu} Let $0\le j<k\le n-1$, $\fr {t} \in \Gam_H (n,j)$,  $\fr {z} \in \Gam_H (n,k)$. If
 $ \fr {t}  \subset  \fr {z}$ and $||\fr {z}|| >a>0$, then $||\fr {t}||>a$.
\end{lemma}
\begin{proof} This statement is intuitively obvious. Below we prove it for the sake of completeness.
Let $\t =\pi^{-1} (\fr {t}) \in  \Gam_B (n,j)$, $\z =\pi^{-1} (\fr {z}) \in  \Gam_B (n,k)$. If   $ \fr {t}  \subset  \fr {z}$, then
$ \t  \subset  \z$, and the inequalities $||\fr {z}|| >a $ and $||\fr {t}||>a$ are equivalent to $|\z| > \th \,a $ and  $|\t| > \th \,a $, respectively. We recall that
the parametrization $\t= \t(\xi, u)$ and $ \z= \z(\eta, v)$ (see (\ref {which}), (\ref {which1}))
implies that $ \t  \subset  \z$ if and only if $\xi \subset \eta$ and $u=v+y$ for some $y \in \xi^\perp \cap  \, \eta$. Hence
\[|\t|\equiv |u|=\sqrt {|v|^2+ |y|^2}\ge|v|\equiv |\z|.\]
It follows that $ \t  \subset  \z$ implies $|\t| \ge |\z|$. Hence, if $|\z|> \th \,a $, then   $|\t| > \th \,a $, which gives the result.
\end{proof}
\begin{corollary} \label{qya1}  Let $0\le j<k\le n-1$, $a>0$. Suppose that  $h$ is a function on $\hn$ satisfying
\be\label{kwu88}  \intl_{||x|| >a} |h (x)|\, \frac{dx}{x_{n+1}^{n-j}}<\infty,    \qquad ||x||= d (x, \tilde o). \ee
If $(R_j h)(\fr {t})=0$ for almost all $\fr {t}$ in the domain $ ||\fr {t}|| >a$, then
\[(R_{j,k} R_j h)(\fr {z}) =0\] for almost all $\fr {z}$ satisfying $ ||\fr {z}|| >a$.
\end{corollary}

Note that the assumption (\ref {kwu88}) is sufficient for a.e. convergence of the Radon transform $R_j h$. It is also necessary if $h$ is zonal; cf.  \cite [Theorem 4.5] {Ru13}.

\begin{theorem} \label{qqwwa1} {\rm (Support Theorem)} Let $0\le j<k\le n-1$, and suppose that  $f \!=\! R_j h$, for some function $h$ on $\hn$ satisfying (\ref {kwu88}).
If $(R_{j,k}f)(\fr {z}) =0 $ for almost all  $\, \fr {z} \in \Gam_H (n,k)$ satisfying  $ ||\fr {z}|| >a$, then $f(\fr {t})=0$ for almost all $\fr {t} \in \Gam_H (n,j)$ with $ ||\fr {t}|| >a$.
\end{theorem}
\begin{proof} The result is a consequence of  the support theorem for   Radon transforms over $(n-1)$-geodesics; see \cite [Theorem 1.6, p. 119] {H11},   \cite [Theorem 6.54] {Ru15}.
 Indeed, by Lemma \ref{esgu},
\[
(R_{j,k}f)(\fr {z})= (R_{j,k} R_j h)(\fr {z}) =(R_k h)(\fr {z}),\qquad \fr {z}\in \Gam_H (n,k). \]
Hence, for any $(n-1)$-geodesic $Z$,
\be\label{k2u88}  \intl_{\fr {z} \subset  Z} (R_{j,k}f)(\fr {z})\, d m(\fr {z}) \!=\!  \intl_{\fr {z} \subset  Z} (R_k h)(\fr {z}) \, d m (\fr {z}) \!=\! (R_{n-1} h)(Z),\ee
where $ d m (\fr {z})$ is the relevant canonical measure. Here the last equality  holds by Lemma \ref{esgu} with $j$ replaced by $k$ and $k$ by $n-1$. If $(R_{j,k}f)(\fr {z}) =0 $
 a.e. when  $ ||\fr {z}|| >a$, then $(R_k h)(\fr {z})$  enjoys this property, too, and, by Corollary \ref{qya1}  (again,  with $j$ replaced by $k$ and $k$ by $n-1$), we obtain that $(R_{n-1} h)(Z)=0$
 for almost all  $ Z$ with $||Z || >a$.
By the support theorem for $R_{n-1}$  (see \cite [Theorem 6.54] {Ru15}),  it follows that
$h(x) =0$  for almost all $x$ satisfying $||x || >a$. But then $f  (\fr {t})= (R_j h)  (\fr {t})=0 $ for almost all $\fr {t} $ with $||\fr {t}||>a$.
\end{proof}

\section {The Projective Model}

In this section we briefly discuss one more realization of the real hyperbolic space, which was not considered in \cite {CFKP} but is definitely  worth being investigated.
We call it the {\it projective model} by analogy with many other similar concepts in geometry.

Let $\nu$ denote a one-to-one map that assigns to each point $x\in \hn$ a one-dimensional linear subspace of $\bbe^{n,1}$ containing this point. The set
 $\Pi_n=\nu (\hn)$ can be regarded as a  model of the $n$-dimensional real hyperbolic space or a projective modification of $\hn$.
The map $\nu$ obviously extends to totally geodesic submanifolds of $\hn$. We denote
\bea\label {assi} \Gam_\Pi (n,d)\!\!\!&=&\!\!\!\nu (\Gam_H (n,d))\\
\!\!\!&=&\!\!\!\{ \t_0 \in G_{n+1, d+1}:  \t_0 =\nu (\fr {t}) \; \text{\rm for some}\;\fr {t} \in \Gam_H (n,d)\}. \nonumber\eea
The latter coincides with the ``Grassmannian ball''
 \be\label {assit1}
\{ \t_0 \in G_{n+1, d+1}:\,  |\t_0|<\pi/4\};\ee
 cf. (\ref{mrrtr}).
 If $\t_0\in \Gam_\Pi (n,d)$,  then the notation  $d\t_0$  will be used for the restriction onto $\Gam_\Pi (n,d)$ of the standard Haar probability measure on $G_{n+1, d+1}$.

 The connection between the spaces of $d$-geodesics
 in the  three models ($\hn$, $B_n$, and $\Pi_n$) is illustrated by the following diagram:
\be\label{necti}
\xymatrix{
\t \in \Gam_B (n,d) \ar[r]^{\pi} \ar[dr]_{\mu} &   \Gam_H (n,d) \ni \fr {t} \ar[d]^{\nu} \\
 & \; \Gam_\Pi (n,d) \ni \t_0
}
\ee
Here the map $\pi: \t \mapsto \fr {t}$ is defined by (\ref{mnqe}), $\mu$ by (\ref{nra}), and $\nu$ by (\ref{assi}).

\begin{lemma} Let $|\t_0|$ and $||t||$ be the geodesic distances (\ref{nra1}) and (\ref{aick}), respectively.
Then
 \be\label {assit2}
\intl_{\Gam_H (n,d)}\!\!\! \Phi  (\fr {t})\, d\fr {t}=
 \frac{\sig_n}{\sig_d} \intl_{\Gam_\Pi (n,d)} \!\! \frac {(\Phi \circ \nu^{-1}) (\t_0)}{ (\cos \, 2|\t_0|)^{(n+1)/2}}\, d\t_0,\ee
 \be\label {assit2T}
 \intl_{\Gam_\Pi (n,d)} \!\! \Psi (\t_0) \, d\t_0 = \frac{\sig_d}{\sig_n} \intl_{\Gam_H (n,d)}\!\!\!  \frac{ (\Psi\circ \nu)(\fr {t})\, d\fr {t}}{ (\ch \, 2||t||)^{(n+1)/2}},\ee
provided that either   integral in the corresponding equality exists in the Lebesgue sense.
\end{lemma}
\begin{proof} The calculations below rely on (\ref{adt2T}).
 The  equality (\ref{assit2}) follows from (\ref{kred1}) and (\ref{nra3X}):
 \bea
\intl_{\Gam_H (n,d)}\!\!\! \Phi  (\fr {t})\, d\fr {t}&=&\!\!\! \intl_{\Gam_B (n,d)} \frac{ (\Phi\circ \pi)(\t)}{(1-|\t|^2)^{(n+1)/2}}\, d\t\nonumber\\
&=&\frac{\sig_n}{\sig_d} \intl_{\Gam_\Pi (n,d)} \!\! \frac {(\Phi\circ \pi)(\mu^{-1}\t_0)\, d\t_0 }{(\cos \,|\t_0|)^{n+1} (1-\tan^2 |\t_0|)^{(n+1)/2}},\nonumber\eea
which gives the result. To prove (\ref{assit2T}), it suffices to set  \[\Psi (\t_0)= (\Phi \circ \nu^{-1}) (\t_0)/(\cos \, 2|\t_0|)^{(n+1)/2}\] in (\ref {assit2}) and change variables.
\end{proof}

Let us proceed with the Radon  transforms.
 Given a function $F$ on  $\Gam_\Pi (n,d)$,  we  define its extension to all $ \t_0  \in G_{n+1, d+1} $ by setting $\tilde F (\t_0) =F (\t_0)$ if
 $|\t_0 |<\pi/4$ and  $\tilde F (\t_0) =0$  if  $|\t|\ge \pi/4$.  For $ 0 \le j<k\le n-1$, the corresponding {\it projective chord transforms} will be defined by

\be\label {Plchyasz} (R_\Pi F)(\z_0)  = \intl_{\t_0 \subset \z_0} f(\t)\, d_{\z_0} \t_0\equiv (\rho_\Pi R_0 \tilde F)(\z_0), \quad  \z_0 \in \Gam_\Pi (n,k), \ee
\be\label {Plch5yasz} (R^*_\Pi \Phi)(\t_0)  = \intl_{\z_0 \supset \t_0} \Phi(\z_0) \,d_{\t_0} \z_0\equiv (\rho_\Pi R^*_0 \tilde \Phi)(\t_0),\quad  \t_0 \in \Gam_\Pi (n,j), \ee
where $R_0$ and $R^*_0$ are the Grassmannian Radon transforms  (\ref{trou}) and (\ref{trou1}), and   $\rho_\Pi$ stands for the restriction to the corresponding  subspaces  meeting $B_n$.

The following statement mimics Theorem \ref{trou6}.
\begin{theorem} \label {Prtrou6}  Let $ M_0$, $ N_0$, $P_0$, and $Q_0$ have the same meaning as in (\ref{trou2})-(\ref{trou5}), respectively. Then
 \be \label {Prtrou7}
R_B F = N_0  R_{\Pi} M_0 F, \qquad R_{B}^* \Phi =Q_0 R_{\Pi}^* P_0 \Phi,\ee
provided that these integrals exist in the Lebesgue sense.
\end{theorem}

Using this theorem, properties of $R_B$-transforms from Subsection \ref{mens}  and $R_0$-transforms from Subsection \ref{plic}  can be reformulated in terms of $R_{\Pi}$-transforms and vice versa.
For instance, Theorem \ref{tally} yields the following statement.

\begin{theorem}\label {tally1}  If $ 0 \le j<k\le n-1$, $j+k\le n-1$, then the operator $R_{\Pi}$ is injective on $L^1 (\Gam_\Pi (n,j))$.
 \end{theorem}

 In fact, this statement is a consequence of Theorem \ref{Ntrou6}; cf. the proof of Theorem \ref{tally}.

Let us return to the hyperboloid model $\hn$.
We recall that by (\ref{afor}) and (\ref{Prtrou7}), $R_H f =NR_B Mf$ and $R_B F = N_0  R_{\Pi} M_0 F$. Combining these formulas, we obtain
$R_H f =N N_0  R_{\Pi} M_0Mf$ or $R_H f = N_1  R_{\Pi} M_1 f$, where
\[ M_1 =M_0M, \qquad N_1=N N_0.\]
 Operators $M_1$ and $N_1$ can be easily computed using the corresponding formulas for $M,N,M_0, N_0$; see
 (\ref{lity}), (\ref{lity1}), (\ref{trou2}),  (\ref{trou3}).

Similar formulas for the dual transforms, namely,  $R_H^*\vp =QR^*_B P \vp$ and
  $R_{B}^* \Phi =Q_0 R_{\Pi}^* P_0 \Phi$ (see (\ref{lity2q}), (\ref{Prtrou7})), yield
 $R_H^*\vp =QQ_0 R_{\Pi}^* P_0 P \vp$ or $R_H^*\vp =Q_1 R_{\Pi}^* P_1 \vp$, where
 \[P_1=P_0 P, \qquad Q_1=QQ_0.\]
  The expressions for $P,Q,P_0, Q_0$ can be found in
 (\ref{lityd}), (\ref{lity1d}), (\ref{trou4}), and (\ref{trou5}), respectively. Hence
  elementary calculations yield the following result.

 \begin{theorem} \label {PrXde}  Let
 \bea \label {Etrou2}
 (M_1 f)(\t_0)&=& \frac{\sig_k}{\sig_j}\, (\cos \, 2|\t_0|)^{-(k+1)/2} (f\circ \nu^{-1})(\t_0), \\
    \label {Etrou3} (N_1 \vp)(\fr {z})&=& (\ch \,2 ||\fr {z}||)^{-(j+1)/2} (\vp \circ \nu)(\fr {z}), \\
 \label {Etrou4} (P_1 \vp)(\z_0)&=&(\cos \, 2|\z_0|)^{(j-n)/2}  (\vp\circ \nu^{-1})(\z_0), \\
 \ \label {Etrou5} (Q_1 \Phi)(\fr {t})&=& (\ch \,2 ||\fr {t}||)^{(k-n)/2}(\Phi \circ \nu)(\fr {t}). \eea
Then
  \be \label {EPrtrou7}
R_H f = N_1  R_{\Pi} M_1 f, \qquad R_H^*\vp =Q_1 R_{\Pi}^* P_1 \vp,\ee
provided that the right-hand sides exist in the Lebesgue sense.
\end{theorem}
\begin{proof} Let us prove (\ref{Etrou2}). The proof of (\ref{Etrou3})-(\ref{Etrou5}) is similar. By (\ref{lity}) and  (\ref {trou2}),
\bea (Mf)(\t)&=&(1-|\t|^2)^{-(k+1)/2} \,(f\circ \pi)(\t), \nonumber\\
 (M_0 F)(\t_0) &=& c\, (\cos |\t_0|)^{-k-1} (F\circ \mu^{-1})(\t_0), \qquad c=\sig_k / \sig_j. \nonumber\eea
 Hence
\bea
 (M_1 f)(\t_0) &=& (M_0 M f)(\t_0)=  c\, (\cos |\t_0|)^{-k-1} ((M f)( \mu^{-1}(\t_0))\nonumber\\
&=&  c\, (\cos |\t_0|)^{-k-1}   (1\!-\!| \mu^{-1}(\t_0)|^2)^{-(k+1)/2} (f\circ \pi)( \mu^{-1}(\t_0))\nonumber\\
&{}& \text {\rm (recall that by (\ref{nra2}),  $|\mu^{-1}(\t_0)| =|\t| =\tan  |\t_0|$)}\nonumber\\
&=&  c\, (\cos |\t_0|)^{-k-1} (1-\tan^2 |\t_0|) ^{-(k+1)/2} (f\circ \nu^{-1})(\t_0)\nonumber\\
&=& c\,(\cos \, 2|\t_0|)^{-(k+1)/2} (f\circ \nu^{-1})(\t_0).\nonumber\eea
\end{proof}

The next theorem is the converse of the previous one.  We observe that
 \bea \label {Etrou2a}
 (M_1^{-1} F)(\fr {t})&=& \frac{\sig_j}{\sig_k}\, (\ch \,2 ||\fr {t}||)^{-(k+1)/2} (F\circ \nu)(\fr {t}), \\
    \label {Etrou3a} (N_1^{-1} \Phi)(\z_0)&=& (\cos \, 2|\z_0|)^{-(j+1)/2} (\Phi \circ \nu^{-1})(\z_0), \\
     \label {Etrou4a} (P_1^{-1} \Phi)(\fr {z})&=&(\ch \,2 ||\fr {z}||)^{(j-n)/2}  (\Phi\circ \nu)(\fr {z}), \\
 \label {Etrou5a} (Q_1^{-1} \Phi)(\fr {t})&=&               (\cos \, 2|\t_0|)^{(k-n)/2}(\vp \circ \nu^{-1})(\t_0). \eea
These equalities can be easily obtained using (\ref{adt2T}).

 \begin{theorem} \label {PrXdea} We have
  \be \label {EPrtrou7a}
R_{\Pi} F = N_1^{-1} R_H M_1^{-1} F, \qquad R_{\Pi}^* \Phi = Q_1^{-1}R_H^*P_1^{-1} \Phi,\ee
provided that the right-hand sides exist in the Lebesgue sense.
\end{theorem}

\section{Conclusion and Open Problems}

The aim of the paper was to establish connections between  higher-rank geodesic Radon transforms in Euclidean, elliptic, and hyperbolic spaces and increase our knowledge about these transforms in the non-smooth setting.
The list of topics studied in the paper is far from being complete, and many interesting developments are expected in the future. Some open
problems have been stated in Remark \ref{owl}, Conjectures  \ref{don} and  \ref{eed4h}.
 It might be worth studying higher-rank Radon transforms in the framework of the Poincar\'e ball model and  the half-space model of the hyperbolic space. Probably some researchers will be inspired to try their hand at higher-rank horospherical (or horocycle) transforms.

Another possible research direction is related to analytic families of higher-rank cosine and sine transforms in the hyperbolic space. These transforms  in the rank-one case were introduced in \cite{Ru02c}. A theory of such transforms in the higher-rank setting is an intriguing and  challenging open problem.  References to such transforms on Grassmann and Stiefel manifolds can be found in \cite{Ru20} and \cite{Zh}. This topic is intimately connected with developments in \cite {GR} and \cite{Zha1} related to explicit inversion of Radon transforms on Grassmannians in terms of G$\mathring{\rm a}$rding-Gindikin integrals. We conjecture that similar theories can be developed in the hyperbolic setting.

\section {Appendix. The Erd\'{e}lyi--Kober Type
Fractional Integrals}\label {kuku}

 We recall some elementary facts from Fractional Calculus \cite[Subsection 2.6.2]{Ru15}, \cite{SKM}.
The following Erd\'{e}lyi--Kober type
fractional integrals on $\bbr_+ =(0, \infty)$ of order $\a >0$ arise in numerous integral-geometric considerations:
\bea
\label{as34b12}%
(I^{\a}_{+, 2} f)(t)
&=&\frac{2}{\Gam
(\a)}\intl_{0}^{t} (t^{2} -r^{2})^{\a-1}f (r) \, r\, dr,\nonumber\\
\label{eci}
(I^{\a}_{-, 2} f)(t)
&=&\frac{2}{\Gam
(\a)}\intl_{t}^{\infty}(r^{2} - t^{2})^{\a-1}f (r) \, r\,
dr.\quad
\nonumber\eea

\begin{lemma}
\label{lifa2}  {\rm \cite[p. 65]{Ru15}} Let    $\a >0$.\

\textup{(i)} The integral $(I^{\a}_{+, 2} f)(t)$ is absolutely
convergent for almost all $t>0$ whenever $r\mapsto rf(r)$ is a locally
integrable function on $\bbr_{+}$.

\textup{(ii)} If
\begin{equation} %
\label{for10z}
\intl_{a}^{\infty}|f(r)|\, r^{2\a-1}\, dr <\infty,\qquad a>0,
\end{equation}
then $(I^{\a}_{-, 2} f)(t)$ is finite for almost all $t>a$. If
$f$ is non-negative, locally integrable on $[a,\infty)$, and
(\ref{for10z}) fails, then $(I^{\a}_{-, 2} f)(t)=\infty$ for every
$t\ge a$.
\end{lemma}

 The corresponding Erd\'{e}lyi--Kober fractional derivatives are defined as the
left inverses ${\Cal D^{\a}_{\pm, 2} = (I^{\a}_{\pm,
2})^{-1}}$. For example, if
$\alpha= m + \alpha_{0}$, $0 \le\alpha_{0} < 1$, $m = \lfloor \alpha \rfloor$, the integer part of $\a$,
then, formally,
\begin{equation} %
\label{frr+z}
\Cal D^{\a}_{\pm, 2} \vp=(\pm D)^{m +1}\, I^{1 - \alpha_{0}}_{
\pm, 2}\vp, \qquad D=\frac{1}{2t}\,\frac{d}{dt}.
\end{equation}
More precisely,  the following statements hold.

\begin{theorem} \label{78awqe555}{\rm (cf. \cite [formula (2.6.22)]{Ru15})}  Let $\vp= I^{\a}_{+, 2} f$, where $rf(r)$ is locally integrable on $\bbr_{+}$.
Then $f(t)= (\Cal D^{\a}_{+, 2} \vp)(t)$ for
almost all $t\in\bbr_{+}$, as in (\ref{frr+z}).
\end{theorem}

\begin{theorem}
\label{78awqe}  {\rm \cite[Theorem 2.44]{Ru15}} If $f$ satisfies (\ref{for10z})
for every $a>0$ and
$\vp\!= \!I^{\a}_{-, 2} f$, then $f(t)= (\Cal D^{\a}_{-, 2} \vp)(t)$ for
almost all $t\in\bbr_{+}$, where $\Cal D^{\a}_{-, 2} \vp$ can be represented as follows.

\noindent
\textup{(i)} If $\a=m$ is an integer, then
\begin{equation} %
\label{90bedr}
\Cal D^{\a}_{-, 2} \vp=(- D)^{m} \vp,
\qquad D=\frac{1}{2t}\,\frac{d}{dt}.
\end{equation}

\noindent
\textup{(ii)} If $\alpha= m +\alpha_{0}, \; m = \lfloor \alpha \rfloor, \; 0 <
\alpha_{0} <1$, then
%
\begin{equation} %
\label{frr+z33}
\Cal D^{\a}_{-, 2} \vp= t^{2(1-\a+m)} (- D)^{m +1}
t^{2\a}\psi, \quad\psi=I^{1-\a+m}_{-,2} \,t^{-2m-2}\,
\vp.
\end{equation}
In particular, for $\a=k/2$, $k$ odd,
\begin{equation} %
\label{frr+z3}
\Cal D^{k/2}_{-, 2} \vp= t\,(- D)^{(k+1)/2} t^{k}I^{1/2}_{-,2}
\,t^{-k-1}\,\vp.
\end{equation}
\end{theorem}

In the above theorem, powers of $t$ are interpreted as the corresponding multiplication operators.

Theorems \ref{78awqe555} and \ref{78awqe} can be used for explicit inversion of diverse Radon-like transforms of radial (or zonal) functions.



\end{document}